\documentclass[reqno]{article}%{amsart} %%%
\usepackage{amsmath,amsthm,amsfonts,amssymb}
\usepackage{paralist}
\usepackage{graphicx,tikz}
\usepackage{verbatim}
\usepackage{epsfig}
\usepackage{epstopdf}
\usepackage{color}
\usepackage{subcaption}

\usepackage[colorlinks=true]{hyperref}

\theoremstyle{plain}
\newtheorem{theorem}{Theorem}[section]
\newtheorem{lemma}{Lemma}[section]

\newtheorem{proposition}{Proposition}[section]

\theoremstyle{definition}

\theoremstyle{remark}
\newtheorem{remark}{Remark}[section]
\newtheorem{example}{Example}[section]

\renewcommand{\div}{{\rm div} }

\numberwithin{equation}{section}

\begin{document}

%\maketitle
%\thispagestyle{empty}
%\begin{center}
%\includegraphics[scale=0.7]{logo.png}

%\vspace{2cm}

%{\huge Technical report for \textsc{FluidVarVisc} project}\\

%\end{center}

%\newpage

\title{Spherically symmetric collapsing solution in the form of shadow wave}
\author{Marko Nedeljkov and Sanja Ru\v{z}i\v{c}i\'{c}}

\date{}
\maketitle

\begin{abstract}
This paper deals with isothermal Euler-Poisson system which is used to model collapse of self-gravitating Newtonian star. Density dependent viscosity term is added on the right-hand side of momentum equation and it has been proved that there exists stable shadow wave solution with unbounded density at the origin. This results is extended to the vanishing pressure case.
\end{abstract}

\section{Introduction}

In the recent years the authors dedicated great attention to the problem of finding singular (unbounded) solutions to the multidimensional compressible Euler and Navier-Stokes equations. The important results are obtained under the assumption that  related flow is cylindrical ($N=2$ below) or spherical ($N=3$), while the analysis of general case brings a lot of open questions. For example, the authors in \cite{imp1, imp2} deal with spherical solution to compressible Euler and Navier-Stokes system, where they proved  that the density will blow up in  finite time. In addition, many authors deal with so called similarity flow since  numerical experiments have shown that  a solution often exhibits such a behavior at least for some classes of initial data (see \cite{Jen} for example).

\medskip

Consider the model
\begin{equation}\label{NDmodel}
\begin{split}
\partial_t \rho +\div (\rho {\bf u})=0,\\
\partial_t (\rho {\bf u})+\div (\rho {\bf u}\otimes {\bf u})=\div T +\rho {\bf f} 
\end{split}
\end{equation}
which describes the motion of Newtonian fluid in $N$ spatial components. The density of the fluid is denoted by $\rho>0$, ${\bf u}$ is velocity, $p=p(\rho)$ pressure, ${\bf f}$ is gravitational force, while $T$ is stress tensor given by
\[T=S-pI,\]
where
\[S=\mu (\nabla {\bf u}+\nabla {\bf u}^t-\frac{2}{N}\div {\bf u}I)+\big(\lambda+\frac{2}{N}\mu\big) \div {\bf u}I\]
is viscous stress tensor. Here $\mu$ denotes shear viscosity, while $\lambda$ is the bulk viscosity coefficient. To ensure physical relevance, we will suppose that coefficients $\mu$ and $\lambda$ depend on density $\rho$.
%{\color{red} Here $\mu$ denotes shear viscosity, while $\zeta$ is the bulk viscosity coefficient given by $\zeta=\lambda+\frac{2}{N}\mu.$ } 

Having in mind that the authors proved the existence of unbounded solution to (\ref{NDmodel}) with $N=2,3$ and $\mu(\rho)=\lambda(\rho)=0$, our aim is to investigate how the introduction of density dependent term $\div T$ from the right hand side of $(\ref{NDmodel})_2$ (called viscosity term in the rest of the paper) influence the behavior of the solution. The function $\mu(\rho)$ can potentially blow up if density becomes unbounded, so the viscosity term should have stabilizing effect. That rises the question if such  term will be strong enough to counteract gravity and prevent blow up of density. Besides that, it is known that  non-monotone density profile can produce a shock in the velocity. Viscosity term could have a smoothing effect in that case and prevent the solutions in the form of shocks.

\medskip

Take $\mu(\rho)=t\rho$ and $\lambda(\rho)=0$ (note that so-called BD entropy condition defined in \cite{BD} is satisfied). In the literature authors usually suppose that $\mu(\rho)=\rho$. However, we are taking $\mu(\rho)=t\rho$ to preserve existence of self-similar solution (see Dafermos 
- DiPerna approach introduced in \cite{DafermosPerna}).

Further, take $N=3.$ 
The stress tensor $T$ is then equal to
\[T=\frac{1}{2}\rho t(\nabla {\bf u}+\nabla {\bf u}^t)-p(\rho)I.\]
%{\color{red}To simplify notation, $i-$th coordinate of  any vector ${\bf y}$ will be denoted by $y_i,$ $i=1,2,3.$}
%\begin{remark}
%Note that in the literature authors usually suppose that $\mu(\rho)=\rho$. The assumption that $\mu(\rho)=t\rho$ is necessary condition to obtain self-similar solution which will be explained later.
%\end{remark}
\medskip

 This model can be used as a model for gravitational collapse, the process which explains structure formation in the universe. Initially, there is a smooth distribution of matter, but after sufficient accretion, the matter may collapse to form pockets of higher density (stars or black holes for example). Even though magnetic field and rotation can play significant role in gravitational collapse, they are often neglected. In the simplified models, it is supposed that the gravity, which tends to draw matter towards the center of the gravity, is the only force that may cause gravitational collapse. However, the star's core is very hot (i.e.\ the energy in the core is very high), which creates pressure within the gas which counteracts the force of gravity. If the gravity overpowers the pressure gradient, gravitational collapse occurs. More precisely, gravitational collapse occurs when there is no more nuclear fuel in the core of the star for fusion process  to  create outward pressure which will counteract the gravitational force.  One of the biggest challenges in the relativistic astrophysics is how to explain and completely understand the final outcome of gravitational collapse. It is believed that the final outcome can be black hole or naked singularity. Roughly speaking, the difference between naked singularity and black hole can be explained in the following way: naked singularity can communicate with faraway observers while black hole cannot. Black hole is surrounded by a boundary layer called event horizon which has a role to ``hide'' it, so nothing that falls inside can get out, while naked singularity does not have such boundary.
 
 % The first naked singularity solution, which we call JMN-1, is formed from the collapse of matter with zero radial pressure, but non zero tangential pressure
 % The second naked singularity solution, which we call JMN-2, is the end state of collapse of a spherical cloud with non-zero radial pressure. It describes, for example, the collapse of a perfect fluid cloud with a locally varying equation of state $k(r)=p/\rho (not strictly isothermal) that approaches a constant value in the neighbourhood of the centre of the cloud. [in Shadows of spherically symmetric black holes and naked singularities]
 
  It should be noted that  it is not completely explained what is the exact role that pressure plays in the end state of gravitational collapse and there are still many open questions about this subject. Spherically symmetric collapse has been extensively studied from both numerical and theoretical point of view (see \cite{astr2, astr1, astr3, Pan}). On the other side, non-spherical collapse gives rise to many open questions due to the lack of good enough models that could explain such phenomena (see \cite{Pan2}).

 In \cite{Hadzgamma} authors proved the existence of collapsing, smooth radially symmetric self-similar solution to the Euler-Poisson system with the pressure satisfying polytropic equation of state
\[p(\rho)=\kappa \rho^\gamma,\quad \gamma \in \Big(1,\frac{4}{3}\Big),\,\kappa>0.\]  
Thus, in this case the density blows up in finite time (see also \cite{Brenner, Weber}). Self-similar gravitational collapse is investigated in the isothermal case when $\gamma=1,$ too (see \cite{Hadz1, shu}). On the other side, it has been proved that it does not exist collapsing solution with finite total mass and energy for $\gamma>\frac{4}{3}$ (see \cite{deng}). In so-called mass-critical case ($\gamma=\frac{4}{3}$), the authors proved existence of collapsing solutions, too (\cite{Gol, Makino}).  There are also papers dealing with vanishing radial pressure case, which is assumption used to model naked singularities as an end state of gravitational collapse (see \cite{Barve,ash, Sha} for more details). Thus, with the introduction of viscosity term in the equations, there are several cases that should be analyzed separately: case $p(\rho)=0$ and the cases when the pressure follows isothermal equation of state $p(\rho)=\kappa \rho,$ $\kappa>0$ or polytropic equation of state with $\gamma\in(1,\frac{4}{3}].$ The case $\gamma>\frac{4}{3}$ is not of interest to us, since it has been proved that no collapsing solution exists, so the introduction of viscosity term is not necessary. It is also possible to consider naked singularity solution as an end state of  collapse of spherical (but not self-similar) cloud with locally varying equation of state
\begin{equation}\label{varying_p}
p(r,\rho)=\kappa(r)\rho^\gamma, \quad \gamma\geq 1,\, \kappa(r)\geq 0. 
\end{equation}
In this case $k(r)$ is not constant but it vanishes in the neighborhood of the centre of the cloud (for more details see \cite{Sha}).

In the recent work, authors in \cite{chen} proved the global existence of finite energy and finite mass solution to the compressible Euler-Poisson system with finite initial energy and finite initial mass using the method of convex integration. The idea behind the work is to approximate the initial system by considering solution to compressible Navier-Stokes-Poisson equations with density dependent viscosity term. Our aim is to approach this problem from the different point of view. Namely, our goal is to prove that there exists collapsing spherically symmetric solution to the system (\ref{3Dmodel}), which can be approximated by shadow wave. 
Shadow wave solutions (SDW for short) are introduced in \cite{mn2010} with an idea to approximate wide class of singular solutions (including the well known delta and singular shocks). The main idea is to introduce small parameter $\varepsilon>0$ which will be used to create a layer around the SDW front in which the solution components could potentially blow up. By taking the distributional  limit when $\varepsilon\to 0$, one obtains the weak solution of the considered system.

 In this paper we mostly focus on  modeling singularities in the origin. It is known that the small changes in initial data can make a big difference in the solution dynamics, so our goal here is to find the conditions that collapsing solution has to satisfy in the outer layers where self-similar solution will be considered. Self-similar solution solves  ODEs system and it is greatly impacted by the initial conditions imposed around the origin. Collapsing solution in the form of shadow wave is not limited to only one class of initial data, and as a consequence, it can be used to model many phenomena connected with gravitational collapse. It is worth mentioning that analysis of self-similar and radially symmetric  solutions, especially if those  solutions are obtained using numerical approximations, is limited and it cannot be applied around origin. The use of shadow wave solves that problem.

 Our aim is  to describe gravitational collapse using shadow waves in two cases: the first one is vanishing pressure case ($p=0$), and the second one is the case when the pressure satisfied isothermal equation of state 
 \[p(\rho)=\kappa \rho, \quad \kappa>0,\]
 where $\sqrt{\kappa}$ is speed of sound.
 Without loss of generality we shall assume that $\kappa=1$, since $\kappa$ can be eliminated from the above equations using simple change of variables. To simplify notation, it will be assumed that pressure equals
 \[p(\rho)=A \rho,\]
 where the constant $A$ is equal to 0 (vanishing pressure case) or 1 (isothermal case).

% we consider the case when the pressure (or pressure gradient) is so small that it can be neglected ($\nabla p=0$). In that case, there is nothing to counteract the gravity, so it is  expected that the matter will accumulate in the center of the gravity which will cause gravitational collapse. We believe that viscosity will have stabilizing effect up to the some degree but it will not be strong enough to prevent gravitational collapse due to the absence of pressure forces.

%This pressure counteracts the force of gravity, putting the star into what is called hydrostatic equilibrium.

%chrome-extension://efaidnbmnnnibpcajpcglclefindmkaj/https://arxiv.org/pdf/gr-qc/0206086
%What role pressures play to determine the final end-state of gravitational collapse?

So, under the assumptions listed above, the system (\ref{NDmodel}) takes the final form
\begin{equation}\label{3Dmodel}
\begin{split}
\partial_t \rho +\div (\rho {\bf u})=0,\\
\partial_t (\rho {\bf u})+\div (\rho {\bf u}\otimes {\bf u})=\div \tilde{T} +\rho {\bf f} ,
\end{split}
\end{equation}
where the symmetric stress tensor is given by
\[\tilde{T}=\frac{1}{2}\rho t(\nabla {\bf u}+\nabla {\bf u}^t)-p(\rho)I=\tilde{S}-p(\rho)I.\]
For the sake of simplicity, we will drop the tilde sign below and it will be supposed  that stress tensor $T$ (and $S$) is obtained by taking $\mu(\rho)=t\rho.$ 
\medskip

Suppose that there exists spherically symmetric solution such that
\[\rho(\vec{x},t)=\rho(r,t),\quad {\bf u}(\vec{x},t)=u(r,t)\frac{\vec{x}}{|\vec{x}|},\]
where $r=|\vec{x}|$.
Then the gravitational force is given by 
\[{\bf f}(\vec{x},t)=f(r,t)\frac{\vec{x}}{|\vec{x}|},\]
where
\[f(r,t)=-\frac{1}{r^2}g(r,t),\quad g(r,t)=\int_0^r s^2 \rho(s,t)ds.\]

Spherically symmetric solution to  (\ref{3Dmodel}) is solution $(\rho(r,t),u(r,t))$ to the system
\begin{equation}\label{rad_sys}
\begin{split}
\partial_t \rho+\partial_r(\rho u)+\frac{2}{r}\rho u &=0\\
% \partial_t(\rho u)+\partial(\rho u^2)+\frac{2}{r}\rho u^2&=2\partial_\rho\mu(\rho)\partial_r \rho %\partial_r u+2\mu(\rho)\partial_r\big(\partial_r u+\frac{2}{r}u\big)+\rho f.
\partial_t(\rho u)+\partial(\rho u^2)+\frac{2}{r}\rho u^2+\partial_r p(\rho)&=2t \Big(\partial_r \rho \partial_r u+\rho\partial_r(\partial_r u+\frac{2}{r}u)\Big)-\frac{1}{r^2}g(r,t)\rho.
\end{split}
\end{equation}
%With $\mu(\rho)=t\rho$ and $\lambda(\rho)=0$, the second equation in (\ref{rad_sys}) reduces to
%\[\partial_t(\rho u)+\partial(\rho u^2)+\frac{2}{r}\rho u^2=2t \Big(\partial_r \rho \partial_r u+\rho\partial_r(\partial_r u+\frac{2}{r}u)\Big)-\frac{1}{r^2}g(r,t)\rho.\]

The paper consists of two  parts. After introducing reader with the system and its physical background,  the authors focus on finding the collapsing spherically symmetric solution in the form of shadow wave in the first part of the paper. That includes  modeling of singularities around the origin and finding the conditions  spherically symmetric solution has to satisfy in the outer layers. In the second part of the paper, we deal with the special case when the appropriate solution in the outer layer is also self-similar. After derivation of ODEs system that such  solution has to satisfy, its behavior has been analyzed and it has been proved that the self-similar solution is bounded in the outer layers. 

\medskip

\section{Collapsing solution in the form of shadow wave}

%{As explained above, in the absence of pressure, the gravitational force will draw the matter inwards to the center of gravity leading to large density.}
 As explained above, the gravitational collapse occurs when the gravitational force overpowers the pressure gradient in the outer layers which causes the matter to accumulate around the  center of gravity. 
 \begin{itemize}
 \item[-] In the absence of pressure ($A=0$), the gravitational force will draw the matter inwards to the center of gravity without  resistance caused by pressure  which acts in the opposite direction. 
 \item[-] It has been  proved that there exists gravitational collapse solution in the isothermal case ($A=1$) for the systems without  density  dependent viscosity term, so it is expected that it is possible to find such a solution when density dependent viscosity term is present in the momentum equation.
 \end{itemize}

Thus, our aim is to prove that there exists spherically symmetric  solution to this problem which is unbounded in the center of the gravity  and it can be approximated by a  shadow wave
\begin{equation}\label{sdw}
(\rho(r,t),u(r,t))=
\begin{cases}
(\rho_\varepsilon(t), u_\varepsilon(t)), & r<\varepsilon t,\\
(\rho_1(r,t),u_1(r,t)), & r>\varepsilon t,
\end{cases}\end{equation}
where $\rho_\varepsilon(t)\to \infty$ as $\varepsilon\to 0$, and $(\rho_1(r,t),u_1(r,t))$ is spherically symmetric solution on its domain such that 
\begin{equation}\label{cd1}
\lim_{\varepsilon\to 0}\frac{\rho_1(\varepsilon t,t)}{\rho_\varepsilon(t)}=0.
\end{equation}
Assumption (\ref{cd1}) is physical since it requires that density is higher in the core than in the outer layers.

The space-time	conic set $|\vec{x}|<\varepsilon t$,
 $\varepsilon>0$ is called the core. It is a self-similar approximation of the origin which is the center of gravity in this model.

 This problem reduces to finding the functions $\rho_\varepsilon(t)$ and $u_\varepsilon(t)$ such that (\ref{sdw}) is approximate spherically symmetric solution to the system (\ref{3Dmodel}). Shadow wave given by (\ref{sdw}) is approximate weak solution to (\ref{3Dmodel}) if the following holds for every test function $\varphi\in C_0^\infty(\mathbb{R}^3\times(0,\infty))$
\begin{equation*}
\begin{split}
I&:=-\int_{0}^\infty \int_{\mathbb{R}^3} (\rho \partial_t \varphi+\rho{\bf u}\cdot \nabla\varphi)dVdt=\mathcal{O}(\varepsilon),\\
J_i&:=-\int_{0}^\infty\int_{\mathbb{R}^3}(\rho u_i\partial_t\varphi+\rho u_i{\bf u}\cdot\nabla \varphi)-T_i\cdot \nabla \varphi+\rho f_i\varphi)dVdt=\mathcal{O}(\varepsilon),\,i=1,2,3,
\end{split}
\end{equation*}
when $\varepsilon\to 0$. 

For the first equation we have
\[\begin{split}
I=&-\int_0^\infty \int_{|\vec{x}|<\varepsilon t} (\rho_\varepsilon\partial_t \varphi+{\bf m}_\varepsilon\cdot \nabla \varphi)dVdt
-\int_0^\infty \int_{|\vec{x}|>\varepsilon t} (\rho_1\partial_t \varphi+{\bf m}_1\cdot \nabla\varphi)dVdt\\
=&\int_0^\infty\int_{|\vec{x}|<\varepsilon t} \partial_t \rho_\varepsilon\varphi dVdt+\int_0^\infty \int_{|\vec{x}|=\varepsilon t}(\rho_1-\rho_\varepsilon)\nu_t \varphi dSdt\\
& + \int_0^\infty \int_{|\vec{x}|=\varepsilon t}({\bf m}_1-{\bf m}_\varepsilon)\cdot\vec{\nu}_x \varphi dSdt+\int_0^\infty \int_{|\vec{x}|>\varepsilon t} (\underbrace{\partial_t\rho_1 +\nabla \cdot {\bf m}_1 }_{=0})\varphi dVdt,
\end{split}\]
where ${\bf m}=\rho {\bf u}$ is momentum variable, while $\vec{\nu}=(\vec{\nu}_x, \nu_t)$ is normal to the SDW front $|\vec{x}-\varepsilon t|=0$ and we have
\[\vec{\nu}_x=\frac{\vec{x}}{|\vec{x}|}\frac{1}{\sqrt{1+\varepsilon^2}},\quad \nu_t=-\frac{\varepsilon}{\sqrt{1+\varepsilon^2}}.\]
To simplify notation, take 
\[\vec{\omega}=\frac{\vec{x}}{|\vec{x}|}.\]
Then
\[{\bf m}(\vec{x},t)=m(r,t)\vec{\omega}, \quad {\bf u}(\vec{x},t)=u(r,t)\vec{\omega}.\]

Since the volume of the ball of radius $\varepsilon t$ is $\frac{4}{3}\pi (\varepsilon t)^3\sim \varepsilon^2$ as $\varepsilon\to 0$ and $\varphi\in C_0^\infty$, we have
\begin{equation}\label{neglect_nabla}
\Big|\int_{|\vec{x}|<\varepsilon t} \nabla \varphi(\vec{0},t)\cdot \vec{x}dV\Big|\leq \int_{|\vec{x}|<\varepsilon t} |\nabla \varphi(\vec{0},t)||\vec{x}|dV\sim \varepsilon^4 \quad \text{as } \varepsilon\to 0,
\end{equation}
so
\[\begin{split}
I_1:=\int_0^\infty\int_{|\vec{x}|<\varepsilon t} \partial_t \rho_\varepsilon(t)\varphi dVdt &\approx \int_0^\infty\int_{|\vec{x}|<\varepsilon t} \partial_t \rho_\varepsilon(t)\big(\varphi(\vec{0},t)+\nabla\varphi(\vec{0},t)\cdot \vec{x}\big) dVdt\\
&= \int_0^\infty\int_{|\vec{x}|<\varepsilon t} \partial_t \rho_\varepsilon(t)\varphi(\vec{0},t) dVdt+O(\varepsilon).
\end{split}\]
$I_1$ will converge if $\rho_\varepsilon(t)=O(\varepsilon^{-3})$. So let us suppose that 
\begin{equation*}\label{assum1}
\rho_\varepsilon(t)=\alpha(t) \varepsilon^{-3},
\end{equation*}
 with $\alpha(t)$ being smooth positive function.
 If $\rho_\varepsilon(t)=o(\varepsilon^{-3})$, then the core mass would disappear, $M(t)=\int_{|\vec{x}|<\varepsilon t} \rho_{\varepsilon}(t)dV\to 0$ as $\varepsilon\to 0$.
So, $I_1$ reduces to 
\[I_1=\int_0^\infty\int_{|\vec{x}|<\varepsilon t} \partial_t (\alpha(t) \varepsilon^{-3})\varphi(\vec{0},t) dVdt+O(\varepsilon)=\frac{4\pi}{3}\int_0^\infty \alpha'(t)t^3\varphi(\vec{0},t)dt+O(\varepsilon).\]
Further, denote by
\[\begin{split}
I_2&:= \int_0^\infty \int_{|\vec{x}|=\varepsilon t}(\rho_1(\vec{x},t)-\rho_\varepsilon(t))\nu_t \varphi(\vec{x},t) dSdt,\\
I_3 &:= \int_0^\infty \int_{|\vec{x}|=\varepsilon t}({\bf m}_1(\vec{x},t)-{\bf m}_\varepsilon(t))\cdot\vec{\nu}_x \varphi(\vec{x},t) dSdt.
\end{split}\]
Using that $\nu_t\approx -\varepsilon$ as $\varepsilon\to 0$ and the assumption on $\rho_\varepsilon(t)$ given above, we obtain
\[\begin{split}
I_2&= -\int_0^\infty \int_{|\vec{x}|=\varepsilon t} (\rho_1(\varepsilon t,t)-\alpha(t)\varepsilon^{-3})\varepsilon(\varphi(\vec{0},t)+\nabla \varphi(\vec{0},t)\cdot \vec{x})dSdt +O(\varepsilon)\\
& = -\int_0^\infty \int_{|\vec{x}|=\varepsilon t} (\rho_1(\varepsilon t,t)-\alpha(t)\varepsilon^{-3})\varepsilon \varphi(\vec{0},t)dSdt +O(\varepsilon).
\end{split}\]
We have used the same kind of reasoning as in (\ref{neglect_nabla}) to neglect the term next to $\nabla \varphi(\vec{0},t)$ above.
The surface area of the ball of radius $\varepsilon t$  is equal to $ 4\pi(\varepsilon t)^2$ and due to condition (\ref{cd1}) (which implies that $\rho_1(\varepsilon t,t)\sim \varepsilon^{-c}$, $c<3$ as $\varepsilon\to 0$), the following holds
\[I_2\approx -4\pi \int_0^\infty (\varepsilon t)^2 \varepsilon (\rho_1(\varepsilon t,t)-\alpha(t)\varepsilon^{-3})\varphi(\vec{0},t)dt= 4\pi \int_0^\infty \alpha(t)t^2 \varphi(\vec{0},t)dt =O(\varepsilon).\]
Again, in the similar manner as above,
 we obtain
\[\begin{split}
I_3&\approx \int_0^\infty\int_{|\vec{x}|=\varepsilon t}(m_1(\varepsilon t,t)-m_\varepsilon(t)) \vec{\omega}\cdot \vec{\nu}_x (\varphi(\vec{0},t)+\nabla \varphi(\vec{0},t)\cdot \vec{x})\,dSdt\\
&\approx  \int_0^\infty\int_{|\vec{x}|=\varepsilon t}(m_1(\varepsilon t,t)-m_\varepsilon(t))  \varphi(\vec{0},t)\,dSdt\\
&=4\pi \int_0^\infty (\varepsilon t)^2 (m_1(\varepsilon t,t)-m_\varepsilon(t))  \varphi(\vec{0},t)\,dt +O(\varepsilon),
\end{split},\]
where we have used that
\[\vec{\omega}\cdot \vec{\nu}_x\approx \vec{\omega}\cdot \vec{\omega}=|\vec{\omega}|^2=1 \quad \text{ as } \varepsilon\to 0.\]
The integral $I_3$ will converge if
\begin{equation}\label{assum2}
m_1(\varepsilon t, t)-m_\varepsilon(t)=u_1(\varepsilon t,t)\rho_1(\varepsilon t, t)-u_\varepsilon(t)\rho_\varepsilon(t)=O(\varepsilon^{-2}),\quad \varepsilon \to 0.
\end{equation}
In that case we get
\[I=I_1+I_2+I_3=\int_0^\infty 4\pi\Big(\frac{1}{3} \alpha'(t)t^3+\alpha(t)t^2+\varepsilon^2t^2 \big(m_1(\varepsilon t,t)-m_\varepsilon(t)\big)\Big)\varphi(\vec{0},t)+O(\varepsilon).\]
Hence,
$I\approx 0$ as $\varepsilon\to 0$ if
\[\lim_{\varepsilon\to 0}\Big(\frac{1}{3} \alpha'(t)t+\alpha(t)+\varepsilon^2 \big(m_1(\varepsilon t,t)-m_\varepsilon(t)\big)\Big)=0.\]
The limit
\[\lim_{\varepsilon\to 0} \varepsilon^2 m_\varepsilon(t)=\lim_{\varepsilon\to 0} \varepsilon^2 \rho_\varepsilon(t)u_\varepsilon(t)\]
exists if 
\[u_\varepsilon(t)\sim \varepsilon^\gamma, \; \gamma\geq 1 \quad\text{ as } \varepsilon\to 0\]
since $\rho_\varepsilon(t)\sim \varepsilon^{-3}$.
On the other side, it is physically reasonable to assume that the velocity is bounded, i.e.\ $u_1(\varepsilon t,t)=\mathcal{O}(1)$, so the limit
\[\lim_{\varepsilon\to 0} \varepsilon^2 m_1(\varepsilon t,t)=\lim_{\varepsilon\to 0} \varepsilon^2 u_1(\varepsilon t,t)\rho_1(\varepsilon t,t)\]
exists if $\rho_1(\varepsilon t,t)=O(\varepsilon^{-2})$. Note that such assumption does not contradict condition (\ref{cd1}). Thus, 
\begin{equation}\label{assum1}
\rho_\varepsilon(t)=\alpha(t)\varepsilon^{-3},\quad u_\varepsilon(t)=\beta(t)\varepsilon^\nu,\,\nu\geq 1,
\end{equation}
with $\alpha(t),\beta(t)$ being  smooth functions such that $\alpha(t)> 0$ and $\beta(t)$ is bounded,
and  
\begin{equation}\label{assum3}
\rho_1(\varepsilon t,t)=O(\varepsilon^{-2}),\quad u_1(\varepsilon t,t)=O(1)
\end{equation}
as $\varepsilon \to 0.$
The integral $I$ will  converge if the following holds
\[0= \begin{cases}
\frac{1}{3} \alpha'(t)t+\alpha(t)+\lim_{\varepsilon\to 0}\varepsilon^2 \rho_1(\varepsilon t,t)u_1(\varepsilon t,t), & \text{ if } \nu>1\\
\frac{1}{3} \alpha'(t)t+\alpha(t)(1-\beta(t))+\lim_{\varepsilon\to 0}\varepsilon^2 \rho_1(\varepsilon t,t)u_1(\varepsilon t,t), & \text{ if } \nu=1
\end{cases}.\]
The sign of $\beta(t)$ determines the sign of velocity $u_\varepsilon$ in the core. Namely, if $\beta(t)<0$, the velocity $u_\varepsilon(t)$ is also negative, which is condition necessary for gravitational collapse to occur. In the case $\nu>1$,  $\beta(t)$ is missing in the above equality, so the behavior of $\beta(t)$ with respect to time does not have any influence on $\alpha(t)$ and vice versa.   On the other side, the assumption $\nu=1$  distinguishes two cases: inflow ($\beta<0$) and outflow ($\beta>0$) of matter from the core.

\begin{remark}
If $\lim\limits_{\varepsilon\to 0}\varepsilon^2 \rho_1(\varepsilon t,t)u_1(\varepsilon t,t)=0$, then we have
\[\frac{1}{3} \alpha'(t)t+\alpha(t)(1-\beta(t))=0\]
for $\nu=1.$
So, if the velocity in  the core is negative, i.e.\ $\beta(t)\leq 0$, then $\alpha'(t)<0$, and the function $\alpha(t)$ decreases with $t$.  If $\beta(t)=0$ for each $t$, then $\alpha(t)$ solves
\[\frac{1}{3}\alpha'(t)t+\alpha(t)=0,\]
so $\alpha(t)=t^{-3}$
and the mass in the core is constant,
\[M_0=\int_{|\vec{x}|<\varepsilon t} \alpha(t)\varepsilon^{-3}dV=\frac{4}{3}\alpha(t)t^3\pi =\frac{4}{3}\pi.\]

%Since we are considering the case when the matter accumulates at the center of the gravity (inflow), the physical condition is $\beta(t)\leq 0$, which corresponds to the negative velocity.
\end{remark}

\medskip

Let us now determine conditions on functions in (\ref{sdw})  to get $J_i\approx 0$, $i=1,2,3$ as $\varepsilon\to 0$.

Having in mind that astronomical body in hydrostatic equilibrium (gravitational force is balanced by internal pressure) will have spherical shape and that any deviations from sphericity would violate the state of hydrostatic equilibrium, it is reasonable to assume that 
\[\nabla p=\rho {\bf f}\]
holds in the ball of radius $\varepsilon t.$ In such a way, the shadow waves are used to model post-core formation stages, when the sphere is initially in hydrostatic equilibrium and the collapse starts in the core. This interpretation is in agreement with result obtained by Shu in \cite{shu}. 

\begin{remark}
The authors in \cite{mn_ober2018} proved the existence of shadow wave solution near the origin for $3\times 3$ compressible Euler system with non-vanishing pressure,
\[p(\rho,e)=\tilde{\kappa}\rho e,\quad \tilde{\kappa}\in(0,2)\]
and ${\bf f}=\vec{0}$.
Therefore, the results obtained here can be generalized for the $3\times 3$ compressible Euler systems. That is left for future work. 
\end{remark}

%{\color{red}
%According to Newton's law of gravity, the net gravitational force on the point mass inside the spherical shell of non-negligible mass is equal to zero, so it is reasonable to  assume  $f_i=0$ for $|\vec{x}|\leq \varepsilon t$. Additionally, we also have $p(\rho)=:p_\varepsilon(t)=0$ for $|\vec{x}|\leq\varepsilon t$ and $\varepsilon$ small enough.}
\noindent
Now, 
\begin{equation}\label{Ji}
\begin{split}
J_i =&\int_0^\infty\int_{|\vec{x}|< \varepsilon t} \partial_t(\rho_\varepsilon(t) u^i_\varepsilon(t))\varphi(\vec{x},t)\,dVdt\\
&+ \int_0^\infty\int_{|\vec{x}|< \varepsilon t} \div(\rho_\varepsilon(t)u^i_\varepsilon(t){\bf u}_\varepsilon(t))\varphi(\vec{x},t)\,dVdt\\
&-\int_0^\infty\int_{|\vec{x}|< \varepsilon t} \big(\div S_\varepsilon^i\underbrace{- \partial_{x_i}p_\varepsilon(t)+\rho_\varepsilon(t)  f_i}_{=0}\big) \varphi(\vec{x},t)\,dV dt\\
&+\int_0^\infty\int_{|\vec{x}|= \varepsilon t} (\rho_1(\varepsilon t,t)u_1^i(\varepsilon t,t)-\rho_\varepsilon(t)u_\varepsilon(t))\nu_t \varphi(\vec{x},t)\,dSdt\\
&+\int_0^\infty\int_{|\vec{x}|= \varepsilon t} (\rho_1(\varepsilon t,t)u_1^i(\varepsilon t,t){\bf u_1}(\varepsilon t,t)-\rho_\varepsilon(t)u_\varepsilon(t){\bf u}_\varepsilon(t))\cdot \vec{\nu}_x \varphi(\vec{x},t)\,dSdt\\
&- \int_0^\infty\int_{|\vec{x}|= \varepsilon t} (S^i_1(\varepsilon t,t)-S^i_\varepsilon(t))\cdot\vec{\nu}_x \varphi(\vec{x},t)\,dSdt\\
&+ \int_0^\infty\int_{|\vec{x}|>\varepsilon t}\big(\underbrace{\partial_t(\rho_1u_1^i)+\div(\rho_1u_1^i{\bf u}_1)-\div T^i_1-\rho f_i}_{=0}\big)\varphi\,dVdt,
\end{split}
\end{equation}
where $u^i$ and $T^i$, $i=1,2,3$ are $i-$th components of the vector ${\bf u}$ and tensor $T$, respectively. To simplify notation, $T_1^i$ is obtained by replacing $u^i$ with $u_1^i$ in 
\[T^i=\frac{1}{2}\rho(\nabla {\bf u}+\nabla {\bf u}^t)_i-p(\rho)\vec{e}_i=S^i-p(\rho)\vec{e}_i,\]
while $T_\varepsilon^i$ is obtained by replacing $u^i$ with $u_\varepsilon^i$ in the same term. The standard base in $\mathbb{R}^3$ is denoted by $\{\vec{e}_1,\vec{e}_2,\vec{e_3}\}$  ($\vec{e}_1=(1,0,0)$, etc.). It is clear that $\rho_\varepsilon(t)$ and ${\bf u}_\varepsilon(t)$ do not depend on $\vec{x}$, so the second and the third integral in (\ref{Ji}) are equal to zero. Let us write $J_i$ as a sum of four integrals,
\[J_i=J_i^1+J_i^2+J_i^3+J_i^4, \quad i=1,2,3,\]
where
\[\begin{split}
J_i^1&:=\int_0^\infty\int_{|\vec{x}|< \varepsilon t} \partial_t(\rho_\varepsilon(t) u^i_\varepsilon(t))\varphi(\vec{x},t)\,dVdt,\\
J_i^2&:=\int_0^\infty\int_{|\vec{x}|= \varepsilon t} (\rho_1(\varepsilon t,t)u_1^i(\varepsilon t,t)-\rho_\varepsilon(t)u_\varepsilon(t))\nu_t \varphi(\vec{x},t)\,dSdt,\\
J_i^3&:=\int_0^\infty\int_{|\vec{x}|= \varepsilon t} (\rho_1(\varepsilon t,t)u_1^i(\varepsilon t,t){\bf u_1}(\varepsilon t,t)-\rho_\varepsilon(t)u_\varepsilon(t){\bf u}_\varepsilon(t))\cdot \vec{\nu}_x \varphi(\vec{x},t)\,dSdt,\\
J_i^4&:=- \int_0^\infty\int_{|\vec{x}|= \varepsilon t} (S^i_1(\varepsilon t,t)-S^i_\varepsilon(t))\cdot\vec{\nu}_x \varphi(\vec{x},t)\,dSdt.
\end{split}\]
Note that the integrals $J_i^k$  can be obtained from $I_k,$ $k=1,2,3$ by replacing $\rho$ with $\rho u^i$. That makes our analysis easier, since the results obtained above can be used here. Thus, replacing $\rho_\varepsilon$ with $\rho_\varepsilon u_\varepsilon^i$ in $I_1$ and using that ${\bf u}^i_\varepsilon(t)=u_\varepsilon(t)\frac{x_i}{|\vec{x}|}$ (or ${\bf u}_\varepsilon(t)=u_\varepsilon(t)\vec{\omega}$), we obtain
\[\begin{split}
J_i^1&=\int_0^\infty\int_{|\vec{x}|<\varepsilon t} \partial_t (\rho_\varepsilon(t)u^i_\varepsilon(t))\varphi(\vec{0},t) dVdt+O(\varepsilon)\\
&=\int_0^\infty\int_{|\vec{x}|<\varepsilon t} \partial_t (\rho_\varepsilon(t)u_\varepsilon(t))\frac{x_i}{|\vec{x}|}\varphi(\vec{0},t) dVdt+O(\varepsilon)\\
&\leq \frac{4\pi}{3}\int_0^\infty \partial_t(\alpha(t)\beta(t))t^3\varepsilon^\nu\varphi(\vec{0},t)dt+O(\varepsilon).
\end{split}\]
Since $\nu\geq 1,$ we conclude that $J_i^1$ is negligible for small $\varepsilon$, i.e.\ $J_i^1=O(\varepsilon)$ as $\varepsilon\to 0$.

\medskip

%{\color{red} Mada mozda bi zbog obja\v snjenja gore trebalo posmatrati samo slu\v caj $\nu=1$.}

\medskip

Similarly, we have
\[\begin{split}
J_i^2&=\int_0^\infty \int_{|\vec{x}|=\varepsilon t}(\rho_1(\vec{x},t)u_1^i(\vec{x},t)-\rho_\varepsilon(t)u_\varepsilon^i(t))\nu_t \varphi(\vec{x},t) dSdt\\
&=\int_0^\infty \int_{|\vec{x}|=\varepsilon t} (m_1(\varepsilon t,t)-m_\varepsilon(t))\omega_i\nu_t\varphi(\vec{0},t)\,dSdt+O(\varepsilon).
\end{split}\]
$J_i^2=\mathcal{O}(\varepsilon)$ follows from $m_1(\varepsilon t,t)-m_\varepsilon(t)=O(\varepsilon^{-2})$ and the fact that $\nu_t\sim -\varepsilon$ as $\varepsilon\to 0$.
Concerning the third non-zero integral in $J_i$, we have 
\[\begin{split}
J_i^3&=\int_0^\infty\int_{|\vec{x}|= \varepsilon t} (\rho_1(\varepsilon t,t)u_1^i(\varepsilon t,t){\bf u_1}(\varepsilon t,t)-\rho_\varepsilon(t)u_\varepsilon(t){\bf u}_\varepsilon(t))\cdot \vec{\nu}_x \varphi(\vec{x},t)\,dSdt\\
&= \int_0^\infty \int_{|\vec{x}|=\varepsilon t} (m_1(\varepsilon t,t)u_1(\varepsilon t,t)-m_\varepsilon(t)u_\varepsilon(t))\omega_i\underbrace{\vec{\omega}\cdot \vec{\nu}_x}_{\sim 1,\,\varepsilon\to 0}\varphi(\vec{0},t)\,dSdt+O(\varepsilon).
\end{split}\]
The assumptions (\ref{assum1}) imply $m_\varepsilon(t)u_\varepsilon(t)=\alpha(t)\beta^2(t)\varepsilon^{2\nu-3},$ $\nu\geq 1$ and
\[\int_0^\infty \int_{|\vec{x}|=\varepsilon t} m_\varepsilon(t)u_\varepsilon(t)\omega_i\vec{\omega}\cdot \vec{\nu}_x\varphi(\vec{0},t)\,dSdt=O(\varepsilon^{2\nu})\to 0 \quad \text{ as } \varepsilon\to 0,\]
so $J_i^3$ converges if
\begin{equation}\label{assumpp}
m_1(\varepsilon t,t)u_1(\varepsilon t,t)=O(\varepsilon^{-2}), \quad \varepsilon\to 0.
\end{equation}
 The condition (\ref{assumpp}) is satisfied if (\ref{assum3}) is. Thus, we have
\[J_i^3=\int_0^\infty \int_{|\vec{x}|=\varepsilon t} m_1(\varepsilon t,t)u_1(\varepsilon t,t)\omega_i \varphi(\vec{0},t)\,dSdt+O(\varepsilon).\]
 All components of  the viscous stress tensor $S_\varepsilon(t)$ (on the SDW front $|\vec{x}|=\varepsilon t$  are equal to zero due to the fact that ${\bf u}_\varepsilon(t)$ depends only on $t$. Hence, the integral $J_i^4$ reduces to
\[J_i^4=-\int_0^\infty \int_{|\vec{x}|=\varepsilon t} S_1^i\cdot \vec{\nu}_x\varphi(\vec{x},t)\,dSdt.\]
We have
\[\partial_{x_j}u^i(\vec{x},t)=\omega_i \omega_j \partial_r u(r,t)-\omega_i\omega_j\frac{1}{r}u(r,t)+\begin{cases}
\frac{1}{r}u(r,t), & \text{ if } i=j\\
0, & \text{ if } i\neq j,
\end{cases}\]
where $w_i=\frac{x_i}{|\vec{x}|}$ denotes the $i-$th component of the vector $\vec{\omega}=\frac{\vec{x}}{r}$, $r=|\vec{x}|$ and $u(\vec{x},t)=u(r,t)\vec{\omega}.$
Then
\[\begin{split}
S^i&=t\rho \frac{1}{2}\big(\nabla{\bf u}+\nabla {\bf u}^t\big)_i\\
&=t\rho\Big(\omega_i\big(\partial_r u(r,t)-\frac{1}{r}u(r,t)\big)\vec{\omega}+\frac{1}{r}u(r,t)\vec{e_i}\Big).
\end{split}\]

Hence,
\[\begin{split}
J_i^4=&-\int_0^\infty \int_{|\vec{x}|=\varepsilon t} t\rho_1(\varepsilon t,t)\omega_i\big(\partial_r u_1(r,t)-\frac{1}{r}u_1(r,t)\big)\vec{\omega}\cdot \vec{\nu}_x \varphi(\vec{0},t)\,dSdt\\
&- \int_0^\infty \int_{|\vec{x}|=\varepsilon t} t\rho_1(\varepsilon t,t)\frac{1}{r}u_1(\varepsilon t,t)\vec{e_i}\cdot \vec{\nu}_x\varphi(\vec{0},t)\,dSdt\\
%&+ \int_0^\infty \int_{|\vec{x}|=\varepsilon t} A\rho_1(\varepsilon t,t)\vec{e_i}\cdot \vec{\nu}_x\varphi(\vec{0},t)\,dSdt+O(\varepsilon)\\
&=-\int_0^\infty \int_{|\vec{x}|=\varepsilon t} \rho_1(\varepsilon t,t)\omega_i t\partial_r u_1(r,t) \varphi(\vec{0},t)\,dSdt+O(\varepsilon),
\end{split}
\]
since $\vec{\omega}\cdot \vec{\nu}_x\approx 1$ and $\vec{e}_i\cdot \vec{\nu}_x\approx \omega_i$ as $\varepsilon\to 0.$ It follows that $J_i^4$ converges if 
\begin{equation}\label{assum4}
\rho_1(\varepsilon t,t)t\partial_r u_1(r,t)=O(\varepsilon^{-2}),\quad \varepsilon\to 0.
\end{equation}
Thus, a shadow wave given by (\ref{sdw}) is an approximate spherically symmetric  solution to the system (\ref{3Dmodel}) if the assumptions (\ref{assum1}), (\ref{assum3}), (\ref{assum4}) are satisfied together with
\[J_i^3+J_i^4=\int_0^\infty \int_{|\vec{x}|=\varepsilon t} \rho_1(\varepsilon t,t)\big(u_1^2(\varepsilon t,t)-t\partial_r u_1(r,t)\big) \omega_i \varphi(\vec{0},t)\,dSdt=O(\varepsilon).\]
The above relation holds if 
\[\lim_{\varepsilon\to 0}\big(u_1^2(\varepsilon t,t)-t\partial_r u_1(\varepsilon t,t)\big)=0.\]
The assumption $u_1(\varepsilon t,t)=O(1)$  implies 
\begin{equation}\label{fake_u_1}
u_1^2(\varepsilon t,t)\approx t\partial_r u_1(\varepsilon t,t)=O(1),\quad \varepsilon\to 0.\end{equation}

If $u_1$ is self-similar solution such that $u_1(r,t)=V(y),$ where $y=\frac{r}{t}$ (see Section \ref{sec:ss}), then (\ref{fake_u_1}) implies

\[V^2(\varepsilon)-t\frac{1}{t}V'(\varepsilon)=V^2(\varepsilon)-V'(\varepsilon)\approx 0,\]
which is not physically reasonable assumption for gravitational collapse, since the equality (\ref{fake_u_1}) implies that  $\partial_r u_1(\varepsilon t,t)$ is non-negative near the core and it does not model inflow of matter towards the center of gravity. 
%In general case ($p\neq 0$), that scenario should not be excluded, since the pressure counteracts the gravity which could lead to positive acceleration.

\medskip

Thus, to obtain physically reasonable spherically symmetric collapsing solution
 to the system (\ref{NDmodel}) one has to impose stronger assumptions on  $u_1(r,t)$ and $\rho_1(r,t)$ at $r=\varepsilon t$. Namely, if we additionally make assumption that 
$
u_1(\varepsilon t,t)=-O(\varepsilon) $ and $ \partial_r u_1(\varepsilon t,t)=-O(1),$
then $J_i^3=O(\varepsilon)$ as $\varepsilon\to 0$. As a consequence, $J_i= O(\varepsilon)$  if and only if $J_i^4=O(\varepsilon)$ as $\varepsilon\to 0$. That implies 
$
\rho_1(\varepsilon t,t)=O(\varepsilon^{-1}),
$ $\varepsilon\to 0.$
\begin{theorem}\label{thm:sdw}
Let $(\rho_1(r,t),u_1(r,t))$ be  solution to the system (\ref{rad_sys}) for $r>\varepsilon t$. The shadow wave solution given by 
\begin{equation}\label{sdw1}
(\rho(r,t),u(r,t))=
\begin{cases}
(\rho_\varepsilon(t), u_\varepsilon(t)), & r<\varepsilon t,\\
(\rho_1(r,t),u_1(r,t)), & r>\varepsilon t,
\end{cases}\end{equation}
is approximate spherically symmetric solution to the system (\ref{3Dmodel}) when $\varepsilon\to 0$ if the following assumptions hold:
\begin{enumerate}
\item The state $(\rho_\varepsilon(t),u_\varepsilon(t))$ is given by  
\begin{equation}\label{ass_vareps}
\rho_\varepsilon(t)=\alpha(t)\varepsilon^{-3},\quad u_\varepsilon(t)=\beta(t)\varepsilon^\nu,\,\nu\geq 1,\quad \varepsilon\to 0,
\end{equation}
where $\alpha(t)>0$ and  $\beta(t)$ are smooth functions such that
\[\frac{t}{3}\alpha'(t)+\alpha(t)(1-\beta(t))=0\]
holds if $\nu=1$ and 
\[\frac{t}{3}\alpha'(t)+\alpha(t)=0\] 
holds if $\nu>1.$  The mass in the core is constant if $\nu>1$.
\item The velocity $u_1(r,t)$ satisfies 
\begin{equation}\label{assum5}
u_1(\varepsilon t,t)=O(\varepsilon), \quad \partial_r u_1(\varepsilon t,t)=O(1),\quad \varepsilon\to 0\end{equation}
while the density $\rho_1(r,t)$ satisfies
\begin{equation}\label{assum6}
\rho_1(\varepsilon t,t)=O(\varepsilon^{-1}),\quad \varepsilon\to 0.
\end{equation}
\end{enumerate}
\end{theorem}

\medskip

Theorem \ref{thm:sdw} gives the conditions that spherically symmetric solution $(\rho_1,u_1)$  has to satisfy on $|\vec{x}|=\varepsilon t$, but it does not require that such a solution is self-similar. However, the assumption  $\mu(\rho)=t\rho$ is introduced to secure existence of self-similar solution.

%\begin{remark}
%It should be noted that some authors believe that the assumption about existence of self-similar solution in the early stages of gas dynamics flow is not correct (see \cite{shu} for example).  
%\end{remark}

\medskip

The above procedure can be applied to the system (\ref{3Dmodel}) without viscosity term ($\mu(\rho)=0$). In that case, the result obtained from the conservation of mass equation is still valid, while we have $J_i^4=0$. Thus, the shadow wave solution in the form (\ref{sdw}) exists if the intermediate state $(\rho_\varepsilon(t), u_\varepsilon(t))$ satisfies (\ref{ass_vareps}) and $J_i^3=O(\varepsilon)$,  i.e.\ 
\[\rho_1(\varepsilon t,t)u_1^2(\varepsilon t,t)=O(\varepsilon^{-1}),\quad\varepsilon\to 0.\] 
So, the result is similar in the case $\mu(\rho)=0$, but now with weaker assumption on $u_1$ at $|\vec{x}|=\varepsilon t$.

\begin{theorem}\label{thm:sdw0}
Let $(\rho_1(r,t),u_1(r,t))$ be  solution to the system
\begin{equation*}
\begin{split}
\partial_t \rho+\partial_r(\rho u)+\frac{2}{r}\rho u &=0\\
\partial_t(\rho u)+\partial(\rho u^2)+\frac{2}{r}\rho u^2+\partial_r p(\rho) &=-\frac{1}{r^2}g(r,t)\rho\\
p(\rho)&=A\rho,\quad A\in\{0,1\}.
\end{split}
\end{equation*}
for $r>\varepsilon t$. 
 The shadow wave solution given by 
\begin{equation*}
(\rho(r,t),u(r,t))=
\begin{cases}
(\rho_\varepsilon(t), u_\varepsilon(t)), & r<\varepsilon t,\\
(\rho_1(r,t),u_1(r,t)), & r>\varepsilon t,
\end{cases}\end{equation*}
is approximate  spherically symmetric solution to the system (\ref{3Dmodel}) with $\mu(\rho)=0$  if the following assumptions hold:
\begin{enumerate}
\item The state $(\rho_\varepsilon(t),u_\varepsilon(t))$ is given by  
\begin{equation*}
\rho_\varepsilon(t)=\alpha(t)\varepsilon^{-3},\quad u_\varepsilon(t)=\beta(t)\varepsilon^\nu,\,\nu\geq 1, \quad \varepsilon\to 0,
\end{equation*}
where $\alpha(t)>0$ and  $\beta(t)$ are smooth functions such that
\[\frac{t}{3}\alpha'(t)+\alpha(t)(1-\beta)(t)=0\]
holds if $\nu=1$ and 
\[\frac{t}{3}\alpha'(t)+\alpha(t)=0\] 
holds if $\nu>1$.
\item The solution $(\rho_1(r,t),u_1(r,t))$ satisfies the condition 
\begin{equation}\label{novisc_cd}
\rho_1(\varepsilon t,t)u_1^2(\varepsilon t,t)=O(\varepsilon^{-1}),\quad\varepsilon\to 0.
\end{equation} 
\end{enumerate}

\end{theorem}
%{\color{red}
\begin{remark}
%{\color{blue}
Note that we do not consider the case when the pressure satisfied polytropic equation of state
\[p(\rho)=\kappa \rho^\gamma,\quad \gamma\in\Big(1,\frac{4}{3}\Big],\,\kappa>0\]
for $r\geq \varepsilon t,$ since it is necessary to impose different conditions on functions $\rho_\varepsilon(t)$ and $u_\varepsilon(t)$ to be able to obtain solution in the form of shadow wave.

% The shadow wave solution in the form (\ref{sdw1}) exists only if pressure gradient equals zero around the origin, like in the case when pressure is given by (\ref{varying_p}). 
 Additionally, it is also possible to obtain naked singularity solution in the form of shadow wave. However, that is left for future work, since the analysis of solution $(\rho_1,u_1)$ requires numerical experiments due to the fact that self-similar solution does not exists in that case.

%as well as the analysis of the $2\times 2$ system with pressure that vanishes near the origin. Such  model can be used to model formation of naked singularities (\cite{Pan2}).

%In this case the assumptions on intermediate state would have to change in order to obtain solution %in the form of shadow wave, since we would need 
%\[(\rho_\varepsilon(t))^\gamma \varepsilon^2=O(1),\quad \varepsilon \to 0\]
%to hold.

%Also,  $\rho_\varepsilon(t)$ would depend on $\gamma.$ Such assumption significantly changes (or even over simplifies) further analysis. Even more, as it will be explained bellow, the introduction of pressure in the system, changes the behavior of solution and makes the analysis of self-similar solution different.}

%{\color{red} --- Ovo svakako treba menjati}
\end{remark}

\medskip

Note that assumptions (\ref{assum5}) and (\ref{assum6}) from Theorem \ref{thm:sdw} or (\ref{novisc_cd}) from Theorem \ref{thm:sdw0} are necessary for shadow wave solution to exists. However, one should have in mind that such condition might not be sufficient, in the sense that not every choice of initial data will give physical collapsing solution. Namely, values of $\rho_1$ and $u_1$ taken at the point $(\varepsilon t,t)$ should agree with already known results (see \cite{shu} for example).

\section{Existence of self-similar solution}\label{sec:ss}

The experiments have shown that  the solution to the system (\ref{rad_sys}) exhibits self-similar behavior in the gravitational collapse at least for some classes of initial data. Due to the fact that there is still many unknowns about the formation of structures in the universe, assumption about self-similar solution is often used.

\begin{theorem}\label{thm:ss}
There exist self-similar solution to the system (\ref{rad_sys}), with $\mu(\rho)=t\rho$ and $\lambda(\rho)=0$, given in the form
\[\rho(r,t)=\frac{1}{t^2}R(y),\; u(r,t)=V(y), \quad y=\frac{r}{t}\]
 if the functions $R(y)$ and  $W(y):=\frac{V(y)-y}{y}$ are solutions to the ODE system
\begin{equation}\label{ode}
\begin{split}
R'Wy =&-R(W'y+3W+1),\\
W''Wy^2 =&\frac{1}{2}(Wy)^2(W'y+W+1-R)+(W'y+1)^2+4W+3W^2\\
& -\frac{A}{2}(W'y+3W+1)
\end{split}
\end{equation}
with initial conditions  
\begin{equation}\label{ode_id}
W(\varepsilon)=\tilde{v}-1,\, W'(\varepsilon)=\frac{\tilde{v}_1-\tilde{v}}{\varepsilon},\, R(\varepsilon)=\frac{d_1}{\varepsilon}, \quad  \tilde{v},\,\tilde{v}_1<0,\, d_1>0, 
\end{equation}
for $\varepsilon>0$ small enough.
\end{theorem}
\begin{remark}
In this paper we choose to consider self-similar solution obtained by taking change of variables $y=\frac{r}{t}$. That means that $y\to 0$ when $t\to \infty$ or $r\to 0$. However, some authors also analyze the self-similar solutions obtained using change of variables  $x=\frac{t}{r^\lambda}$, $\lambda>0$ (see \cite{sim} for example).
\end{remark}
%\begin{proof}
\noindent
{\it Proof of Theorem \ref{thm:ss}.}
Using the change of variables $y=\frac{r}{t}$ and $\rho(r,t)=\frac{1}{t^2}R(y)$, $u(r,t)=V(y)$, we obtain the self-similar form of the system (\ref{rad_sys}). After replacing all variables in the first equation and multiplying the obtained equation with $t^3$, one gets
\[R(V'-2+\frac{2}{y}V)+R'(V-y)=0\]
or
\begin{equation}\label{ss_1st}
R'W+\frac{1}{y}R(W'y+3W+1)=0.
\end{equation}
Using the equivalent form of (\ref{ss_1st}),
\[(RWy^3)'=-Ry^2,\]
 it is easy to transform the gravitational force in the self-similar form. Namely, 
 \[g(r,t)=\int_0^r s^2\rho(s,t)ds=\int_0^y R(a)a^2tda=-t\int_0^y(RWa^3)'da=-tR(y)W(y)y^3,\]
 where $a=\frac{s}{t}$.
 For  smooth solutions, the term 
 \[\partial_t(\rho u)+\partial(\rho u^2)+\frac{2}{r}\rho u^2+\partial_r p(\rho)\]
 from the left-hand side of the momentum equation is equivalent to
 \[\begin{split}
 \rho(\partial_t u+u \partial_r u)+\partial_r p(\rho)&=\frac{1}{t^2}R\frac{1}{t}\big(V'(V-y)\big)+\frac{A}{t^3}R'\\
 &=\frac{1}{t^3}R(W'y+W+1)Wy+\frac{A}{t^3}R'.
 \end{split}\]
 Besides that, the following holds
 \[\begin{split}
 \partial_r(\partial_r u+\frac{2}{r}u)&=\frac{1}{t^2}(W''y+4W'),\\
 \partial_r \rho \partial_r u&=\frac{1}{t^4}R'(W'y+W+1),\\
 -\frac{1}{r^2}\rho g(r,t)&=\frac{1}{t^3}R^2Wy,
 \end{split}\]
 so the right-hand side of the equation $(\ref{rad_sys})_2$ equals
 \[\begin{split}
& 2t \Big(\partial_r \rho \partial_r u+\rho\partial_r(\partial_r u+\frac{2}{r}u)\Big)-\frac{1}{r^2}g(r,t)\rho\\
=& \frac{1}{t^3}\Big(2\big(R(W''y+4W')+R'(W'y+W+1)\big)+R^2Wy\Big).
 \end{split}\]
 Using (\ref{ss_1st}) to eliminate $R'$ from the above terms and multiplying the whole equation by $t^3$, one can easily prove that a self-similar solution to the system (\ref{rad_sys}) satisfies the ODE system (\ref{ode}).

The initial conditions at $y=\varepsilon$ are chosen such that the assumptions (\ref{assum5}) and (\ref{assum6}) of the Theorem \ref{sdw} are satisfied. More precisely, we need
\[u(\varepsilon t,t)=V(\varepsilon)=\tilde{v} \varepsilon, \quad \partial_r u(\varepsilon t,t)=\frac{1}{t}V'(\varepsilon)=\frac{1}{t}\tilde{v}_1\quad \rho(\varepsilon t,t)=\frac{1}{t^2}R(\varepsilon)=\frac{1}{t^2}\frac{1}{\varepsilon}d_1\]
or
\[W(\varepsilon)=\tilde{v}-1, \quad W'(\varepsilon)=\frac{1}{\varepsilon}(V'(\varepsilon)-1-W(\varepsilon))=\frac{\tilde{v}_1-\tilde{v}}{\varepsilon}\]
and $R(\varepsilon)=\frac{d_1}{\varepsilon}.$

The physical conditions for the gravitational collapse solution  near the surface of the core are
\begin{equation}\label{gc}
\tilde{v}<0, \quad\tilde{v}_1<0,\quad d_1>0,
\end{equation}
since the gravitational force will draw the matter to the core.
Note that (\ref{gc}) implies 
\[R'(\varepsilon)=-\frac{d_1}{\varepsilon}\frac{\tilde{v_1}+2(\tilde{v}-1)}{(\tilde{v}-1)\varepsilon}<0\]
which is necessary for gravitational collapse to occur.  %{\color{blue}In order to obtain self-similar solution $(\rho_1,u_1)$ such that shadow wave solution (\ref{sdw1})  is overcompressive, it is sufficient to take $\tilde{v}<-A\leq -1$. However, the assumption $\tilde{v}<0$ is sufficient for theorems stated below to hold.}
\hfill $\Box$
%\end{proof}

%\begin{remark}
%From the proof of Theorem \ref{thm:ss} it is easy to see that the assumption $\mu(\rho)=t\rho$ is sufficient to obtain existence of the self-similar solution to the system (\ref{rad_sys}). Namely, if $\mu(\rho)=\rho$, time variable  $t$ would not vanish  in the obtained equations which would imply that such self-similar solution does not exist. 
%It is possible that introducing different change of variables by defining $R$ and $V$ on the different way would allow us to change the assumption on $\mu(\rho).$ However, up until now, the authors were not able to find such an example.
%\end{remark}

%Due to the standard ODE theory, the solution to the ODEs system (\ref{ode}) exists as long as $W\neq 0.$ 

\medskip

By taking the initial conditions (at $y=\varepsilon$) such that assumptions listed in Theorem \ref{thm:sdw} are satisfied, we obtain the smooth solution to (\ref{ode}) at least for some $y>\varepsilon$. However, the smooth solution will blow up only if $W$ reaches zero at some point, which would mean that velocity  turned from negative to positive (inflow to outflow). 
%That corresponds to outflow of matter outside of center of gravity. 
%We believe that such a scenario is not physical since there is no pressure to draw the matter from the center of gravity. 
Numerical experiments have shown that such a scenario is not possible under the initial conditions given above.
In the rest of the paper we shall prove that solution to do system (\ref{ode}) subject to initial data
\[W(\varepsilon)<-1,\, W'(\varepsilon)<0,\quad R(\varepsilon)>0\]
is bounded for $y>\varepsilon$. Even more, $W(y)<-\frac{1}{3}$ for $y>\varepsilon.$

\begin{remark}
Note that not every choice of  $\tilde{v},\tilde{v}_1<0$ will imply $\tilde{v}_1-\tilde{v}<0$ (which would mean $W'(\varepsilon)<0$). However, the numerical experiments have shown that  in the case $\tilde{v}_1-\tilde{v}>0$ ($W'(\varepsilon)>0$), $W'$ becomes negative after one iteration and continues to behave as in the case $\tilde{v}_1-\tilde{v}<0$. Besides that, the results obtained in the sequel hold under assumption $\tilde{v}_1-\tilde{v}<0$.
\end{remark}

\subsection{Notes on the critical points}

An introduction of viscosity term in the momentum equation, as well as assumption of vanishing pressure can significantly change the behavior of solution to the system (\ref{rad_sys}). When analyzing the self-similar solutions to (\ref{rad_sys}) without viscosity term ($\mu=0$) and assuming that the pressure satisfies the polytropic equation of state,
\[p(\rho)=\kappa \rho^\gamma, \quad \gamma\geq 1,\,\kappa>0,\]
 the authors found that there exists a critical point (also called sonic point) which plays significant role in the analysis of the solution dynamics (see \cite{Hadzgamma,Hadz1,shu}).

 The critical point is the point in which the self-similar solution could potentially break down. To demonstrate, take $\mu=0$ and $p(\rho)=\rho$ in (\ref{3Dmodel}). Suppose that there exists spherically symmetric and self-similar solution to that system. Using the same change of variables as above, one can easily obtain that $(R,W)$ satisfy the following ODE system
 \begin{equation}\label{ode_p}
 \begin{split}
 RW'y&=-R'Wy-R(3W+1)\\
 R'&=-RWy(W'y+W+1)+R^2Wy.
 \end{split}\end{equation}
 The term $R'$ on the left-hand side of the second equation appears due to the  assumption of non-vanishing pressure.
 Expressing $W'y$ from the first equation and replacing it into the second equation, one obtains 
 \[R'=RWy\frac{R+2W}{1-(Wy)^2}.\]
 The point $\bar{y}$ such that
 \[1-W(\bar{y})^2\bar{y}^2=0\]
 is called sonic or critical point.
 It is clear that the classical solution will break down at the point $y=\bar{y}$ if
 \[1-W(\bar{y})^2\bar{y}^2=0,\quad R(\bar{y})+2W(\bar{y})\neq 0.\]
 To secure continuity of the solution, it is necessary to have
 \[R(\bar{y})+2W(\bar{y})=0.\]
 Such a solution is in literature often referred as  Larson-Penston collapsing solution (see \cite{Hadz1, Larson, Penston}). However, only very special initial and boundary conditions will give the Larson-Penston solution. Otherwise, the density blows up while passing through the sonic point. Also, some authors focus their attention  on the solution without critical point (see \cite{shu} for example).
   If $p(\rho)=0$, then the term $R'$ in the second equation in (\ref{ode_p}) vanishes and we obtain
 \[R'=-\frac{R}{Wy}(R+2W).\]
 The only critical point in this case is point $\tilde{y}$ such that $W(\tilde{y})\tilde{y}=0$. Singularity at zero (center of gravity) is analyzed separately (solution in the form of shadow wave at the origin),  while it should be investigated if there exist the solution such that $W(\tilde{y})=0$ for some $\tilde{y}>0$.  It can be proved that at least for some initial data we have $W<0$ for each $y>0$. 
 %corresponds to scenario when gravity is not strong enough, so the matter accumulates away from the center of gravity. 
 
% is not physical. Namely,  due to the absence of pressure, there is no other force to counteract gravity and  draw the matter from the center of gravity ($W>-1$ implies positive velocity), which would correspond to scenario $W(\tilde{y})=0$.

\medskip

The introduction of viscosity term  changes the obtained ODE system (the second derivative of $W$ appears in the second equation), and the critical point which is considered in the case $\mu=0$ with non-vanishing pressure is not the point in which density can blow up, even if the pressure is non-negligible. In our case, the self-similar collapsing solution will break down in finite time and far from the center of gravity only if $W(y^\ast)=0$ and $W'(y^\ast)\neq \frac{1}{3}$ for some $y^\ast>0$. As already said, we shall prove that such scenario is not possible, i.e.\ $W$ will not cross the line $W=0.$

%That scenario corresponds to the outflow of the matter from the center of gravity (velocity is positive) and our aim is to prove that it is not possible, i.e.\ that there exists $C>0$ such that $W(y)\leq -C$.
% and the experiments have shown that the function $W$ in the self-similar solution to (\ref{ode}), with the initial conditions  satisfying the assumptions of Theorem \ref{thm:sdw} will not pass the boundary -1 at any time (i.e.\ $W(y)\leq -1$ for all $y$). %That implies that stabilizing effect of viscosity is not strong enough to prevent the accumulation of the matter in the outer layers.

\medskip

\medskip

\subsection{Self-similar solution to the system with viscosity term}

Let us first analyze the behavior of solution around the origin. Based on Theorem \ref{thm:ss} the initial conditions at $y=\varepsilon$ are (\ref{ode_id}).
%Based on Theorems \ref{thm:sdw} and \ref{thm:ss}, the collapsing solution should satisfy initial conditions
%\[W(\varepsilon)=\tilde{v}-1,\; W'(\varepsilon)=\frac{\tilde{v}_1-\tilde{v}}{\varepsilon},\; R(\varepsilon)=\frac{d_1}{\varepsilon}>0,\quad \tilde{v},\tilde{v}_1<0\] 
%at $y=\varepsilon$ for small $\varepsilon>0.$ 
Then $W''(\varepsilon)<0$, i.e.\ $W$ is concave at $y=\varepsilon,$ while $W'(\varepsilon)$ is decreasing at $y=\varepsilon$. Namely, the term
%\[\big(\omega_1+\frac{1}{y^\ast}\big)^2+\frac{4}{y^\ast}\omega_0(1+\frac{3}{4}\omega_0)\]
\[\frac{1}{W}\Big(\big(W'+\frac{1}{y}\big)^2+\frac{4}{y^2}W(1+\frac{3}{4}W)-\frac{A}{2y^2}(W'y+3W+1)\Big)\]
in $(\ref{ode})_2$ is negative of order $\varepsilon^{-2}$ at $y=\varepsilon$ and it dominates the term
%\[\frac{1}{2}\omega_0^2(\omega_1y^\ast+\omega_0+1-r_0)\]
\[\frac{W}{2}(W'y+W+1-R)\]
which is positive of order $\varepsilon^{-1}$ at $y=\varepsilon$. Additionally, suppose that $\tilde{v}_1-\tilde{v}<0$, i.e.\ $W'(\varepsilon)<0$. As already said, the above initial assumptions imply $R'(\varepsilon)<0.$
Also, at least for some $y>\varepsilon$, we have
\[W(y)<W(\varepsilon)<-1,\quad W'(y)<W'(\varepsilon)<0,\quad R(y)<R(\varepsilon).\]

\medskip

The following lemma will be used in the sequel.
\begin{lemma}\label{comp_thm}
Let $u(t)$ and $v(t)$ be real functions, differentiable and continuous on $[a,b]$ and suppose $u(a)\leq v(a)$. If 
\[u'(t)<v'(t) \quad \text{ for } t\in[a,b],\]
then $u(t)<v(t)$ for $t\in(a,b].$
\end{lemma}
\begin{proof}
 Define $h(t)=u(t)-v(t)$. 
Then $h(a)\leq 0$ and $h'(t)<0$ for $y\in[a,b]$. 
Since functions $u$  and $v$ are continuous, it follows $h(t)<0$ for $t\in(a,b]$.
\end{proof}

\medskip

\begin{proposition}\label{thm:Rvac}
Let $R$ and $W$ be solutions to the ODE system (\ref{ode}) and suppose
\[W(\varepsilon)=\tilde{v}-1,\; W'(\varepsilon)=\frac{\tilde{v}_1-\tilde{v}}{\varepsilon},\; R(\varepsilon)=\frac{d_1}{\varepsilon},\]
where 
$ \tilde{v},\tilde{v}_1<0,d_1>0$ and $\tilde{v}_1-\tilde{v}<0$. There exists $z>\varepsilon$ and positive function $\tilde{R}$ such that $R(\varepsilon)\leq \tilde{R}(\varepsilon)$ and 
\[R(y)<\tilde{R}(\varepsilon)\Big(\frac{y}{\varepsilon}\Big)^{-\big(3+\frac{1}{W(\varepsilon)}\big)}\]
for $y\in(\varepsilon,z)$.
\end{proposition}
\begin{proof}
Take $z>\varepsilon$ such that for $y\in(\varepsilon,z)$,
\[W(y)<W(\varepsilon),\quad W'(y)<W'(\varepsilon).\]
Then
\[R'(y)=-\frac{R(y)}{W(y)y}(W'(y)y+3W(y)+1)<-\frac{R(y)}{W(y)y}(3W(y)+1)  \quad \text{for } y\in(\varepsilon,z).\]
Define the function $\tilde{R}$ such that
\[\tilde{R}'(y)=-\frac{\tilde{R}(y)}{W(\varepsilon)y}(3W(\varepsilon)+1),\quad y\in(\varepsilon,z)\]
and $R(\varepsilon)\leq\tilde{R}(\varepsilon).$
Then
\[\tilde{R}(y)=\tilde{R}(\varepsilon)\Big(\frac{y}{\varepsilon}\Big)^{-\big(3+\frac{1}{W(\varepsilon)}\big)},\quad y\in(\varepsilon,z).\]
Due to Lemma \ref{comp_thm}, it follows 
\[R(y)<\tilde{R}(y),\quad y\in(\varepsilon,z).\]
Since $W(y)<-1$ for $y\in(\varepsilon,z)$, we have $-\big(3+\frac{1}{W(\varepsilon)}\big)<-2,$
so
\[\lim_{\varepsilon\to 0}\tilde{R}(y)=0.\]
That implies $R(y)\to 0$ as $\varepsilon\to 0$ for $y\in(\varepsilon,z).$
\end{proof}

\begin{proposition}\label{thm:H}
Let $W$ be continuous and differentiable function on the interval $(\varepsilon,b)$. 
Suppose $H(y):=W'(y)y+3W(y)+1<0$ for $y\in (\varepsilon,b)$ and $W(\varepsilon)<-\frac{1}{3}$. Then $W(y)< -\frac{1}{3}$ for $y\in(\varepsilon, b).$
\end{proposition}

In order to prove Proposition \ref{thm:H} we will prove the following lemma first.
\begin{lemma}\label{lemma:comp}
Suppose the real functions $u(x)$ and $v(x)$ are continuous and differentiable on the interval $[a,b]$, and $u(a)=v(a)$. Let $f:\mathbb{R}\times \mathbb{R}\to \mathbb{R}$ be continuous function and suppose
\begin{equation}\label{comp_ineq}
\frac{du}{dx}-f(x,u)<\frac{dv}{dx}-f(x,v), \quad x\in[a,b].
\end{equation}
Then $u(x)< v(x)$ for $x\in(a,b].$
\end{lemma}
\begin{proof}
Suppose that there exists $c\in(a,b]$ such that $u(c)\geq v(c)$. The function $u-v$ is continuous on $[a,b]$ and 
\[u(a)-v(a)=0,\quad \frac{du(a)}{dx}-\frac{dv(a)}{dx}<0.\]
It follows that there exists $d\in(a,c]$ such that 
\[u(d)-v(d)=0,\quad \frac{du(d)}{dx}-\frac{dv(d)}{dx}>0,\]
since the negative and decreasing function in the neighborhood of the point $x=a$ has to start to increase to become non-negative at the point $x=c$. That contradicts assumption (\ref{comp_ineq}).
\end{proof}

\medskip

\noindent
{\it Proof of Proposition \ref{thm:H}.}
Let $\tilde{W}$ be a solution to the ordinary differential equation 
\[\tilde{W}'y+3\tilde{W}+1=0,\quad y>\varepsilon \]
with $\tilde{W}(\varepsilon)=W(\varepsilon)$,
i.e.\ $\tilde{W}(y)=\frac{3W(\varepsilon)+1}{3}\Big(\frac{\varepsilon}{y}\Big)^3-\frac{1}{3}$ for $y\geq \varepsilon.$
Define $f(y,W):= -\frac{1}{y}(3W+1).$ Applying Lemma \ref{lemma:comp} on the functions $W$ and $\tilde{W}$, one obtains that if 
\[H(y)=\frac{dW}{dy}-f(y,W)<\frac{d\tilde{W}}{dy}-f(y,\tilde{W})=0, \quad \text{for } y\in(\varepsilon,b),\]
then $W(y)<\tilde{W}(y)\leq -\frac{1}{3}$. That concludes the proof of Proposition \ref{thm:H}. \hfill $\Box$

\bigskip

%Denote $H(y):=W'(y)y+3W(y)+1<0$.
\begin{proposition}\label{thm:Hless}
 Let $(R,W)$ be solution to the system (\ref{ode}) with initial conditions
\begin{equation}\label{id}
W(\varepsilon)<-1,\; W'(\varepsilon)<0,\; W''(\varepsilon)<0, \quad R(\varepsilon)>0, \qquad \varepsilon>0.\end{equation}
 Suppose $W(y)< -\frac{1}{3}$ for $y\in(\varepsilon,c)$ and $R$ is small enough such that $-2W(y)-R(y)>0$ on the interval $(\varepsilon,c)$. Then $H(y)<0$ for each $y\in(\varepsilon,c)$.
\end{proposition}
\begin{proof}
The initial conditions (\ref{id}) imply that there exist $y_c>\varepsilon$ such that $W(y)<W(\varepsilon)<-1$ and $W'(y)<W'(\varepsilon)<0$ at least in the interval $(\varepsilon, y_c)$. As a consequence, $H(y)<0$ for $y\in[\varepsilon,y_c)$. Suppose now that there exist $y_p>y_c$ such that $H(y)<0$ for $y\in[\varepsilon,y_p)$ and $H'(y_p)\geq 0$, $H(y_p)=W'(y_p)+3W(y_p)+1=0$. The assumption $H'(y_p)\geq 0$ is due to the fact that the negative function $H$ has to increase before the point $y=y_p$ in order to reach 0.  

By writing the equation $(\ref{ode})_2$  in equivalent form
\begin{equation*}\label{eq2_equiv}
\begin{split}
(W'y+3W+1)'=&\frac{1}{2}Wy(W'y+W+1-R)\\&+\frac{(W'y+3W+1)(W'y+W+1)}{Wy}
-\frac{A}{2Wy}(W'y+3W+1),
\end{split}
\end{equation*}
one obtains
\begin{equation}\label{y_p}
	H'(y_p)=\frac{1}{2}W(y_p)y_p(-2W(y_p)-R(y_p)).
\end{equation}
So, $H'(y_p)$ is negative if $2W(y_p)+R(y_p)$ is negative, which contradicts the assumption $H'(y_p)\geq 0.$ It follows that $H(y)<0$ as long as $-2W(y)-R(y)>0.$
\end{proof}

In order to prove global existence of self-similar solution for $y>\varepsilon$ it is enough to prove that $R$ is decreasing and $W$ negative for each $y>\varepsilon$ which is proved by combining 	 Propositions \ref{thm:Rvac}, \ref{thm:H} and \ref{thm:Hless}.

\begin{theorem}\label{thm:main_anl}
Let $R$ and $W$ be solutions to the ODE system (\ref{ode}) and suppose
\begin{equation}\label{id_collapse}
W(\varepsilon)=\tilde{v}-1,\; W'(\varepsilon)=\frac{\tilde{v}_1-\tilde{v}}{\varepsilon},\; R(\varepsilon)=\frac{d_1}{\varepsilon},
\end{equation}
where 
$ \tilde{v},\tilde{v}_1<0,d_1>0$ and $\tilde{v}_1-\tilde{v}<0$. Then $W(y)<-\frac{1}{3}$ and $R'(y)<0$ for each $y\geq \varepsilon$.
\end{theorem}
\begin{proof}
The function $R$ will start to increase only if one of two conditions: $W(y)<0$ or $H(y)<0$ is violated (see $(\ref{ode})_1$). 
	Due to Proposition \ref{thm:Rvac}, at least for some $y>\varepsilon$ the function
	$R$ is bounded from above by a function small enough for  condition $-W(y)-R(y)>0$ imposed in Proposition \ref{thm:Hless} to hold. That condition will hold as long as $R$ is decreasing and W negative.  
	%If $W(y_0)>-\frac{1}{3}$ or $R'(y_0)>0$ for some $y_0$, then $H(y_0)>0$ 
	
	Combining Propositions \ref{thm:H} and \ref{thm:Hless} one concludes that for sufficiently small $\varepsilon >0$, $W(y)< -\frac{1}{3}$ if and only if $H(y)<0$ at least for some $y>\varepsilon$.
	 
	Suppose now that there exist $y_k$ such that $W(y)<-\frac{1}{3}$, $H(y)<0$ and $R'(y)<0$ for $y\in(\varepsilon,y_k)$ and  $W(y_k)=-\frac{1}{3}$, $W'(y_k)\geq 0$. 
	Then $H(y_k)=W'(y_k)y_k=0$ and $H'(y_k)\geq 0$ (since $H<0$ for $y<y_k$). $H(y_k)=0$ implies $W'(y_k)=0.$ On the other side, from the proof of Proposition \ref{thm:Hless} we have $H'(y_k)<0$ (see (\ref{y_p})), which is in contradiction with $H'(y_k)\geq 0$. It follows that it does not exist $y_k$ such that $W(y_k)=-\frac{1}{3}$, i.e.\
	$W(y)<-\frac{1}{3}$ for each $y>\varepsilon.$
	
\end{proof}
\medskip

%The function $R$ will blow up only if the function $W$ reaches zero at some point. However, Theorem \ref{thm:main_anl} proves that 
%To conclude, under the initial conditions (\ref{id_collapse}), the self-similar solution to the system (\ref{rad_sys}) is bounded outside the center of gravity, i.e.\
%the density $\rho$ (since $\rho(r,t)=\frac{1}{t^2}R(y)$) will not blow up for $r>\varepsilon t$, $\varepsilon>0$.
 
To conclude, there exist global radially symmetric  solution to the system (\ref{rad_sys}) given by a shadow wave at the origin  and self-similar solution that solves the problem (\ref{ode}, \ref{ode_id}).

\medskip

\begin{remark}
Note that $W>-1$ implies that the velocity $V(y)$ is positive which corresponds to outflow of matter  from the center of gravity. The result of Theorem \ref{thm:main_anl} does not exclude $-1<W<-\frac{1}{3}$ as possibility. %Numerical experiments have shown that for some initial conditions it is possible to obtain solution such that $V$ is positive in some interval far from the center of gravity.
\end{remark}

\medskip

As already said, the initial  smooth distribution of matter gives the the behavior of solution when $y\to \infty$. Namely,
\[\rho(r,t)\to {\rm const},\, u(r,t)\to 0\quad \text{ as } t\to 0^+\]
holds and it is equivalent to
\begin{equation}\label{at_inf}
R(y)\to 0, \, W(y)\to -1 \quad \text{ as } y\to \infty,
\end{equation}
since $y=\frac{r}{t}\to \infty$ when $t\to 0.$  
It is easy to prove the asymptotic behavior of the  solution to the  system (\ref{ode}) when $y\to \infty$. We have
\begin{equation}\label{assym_inf}
R(y)\to 0,\, W(y)\to -1,\, W'(y)y\to 0 \quad \text{as } y\to \infty,
\end{equation}
which corresponds to (\ref{at_inf}).  Note that $R(y)=\frac{2}{y^2}$ and $W(y)=-1$ is exact solution to the system (\ref{ode}) if $A=1$, which is not the case if $A=0$.
%{\color{red}
%Suppose $y^\ast=\varepsilon^{-1}$ with $\omega_0(y^\ast)\to -1^-,$ $r_0(y^\ast)\to 0^+$ as $\varepsilon\to 0$ and $\omega_1(\varepsilon^{-1})>0$. 
%The function $W'(y)y$ is bounded when $y\to \infty$, since otherwise $W''$ would be unbounded and negative. Even more, $W'(y)y\to 0$ as $y\to \infty$

%Also, to ensure boundedness of $W''$ at infinity, suppose $\omega_1(\varepsilon^{-1})\varepsilon=O(\varepsilon)$ (since $\omega_1(\varepsilon^{-1})\varepsilon=O(1)$ would imply that $W$ is not bounded at infinity). It is clear that
%\[r_1(\varepsilon^{-1})\to 0, \quad \varepsilon\to 0\] 
%and
%\[2\omega_2(\varepsilon^{-1})\omega_0(\varepsilon^{-1})\to 0, \quad \varepsilon\to 0.\]}

%Hence, $\omega_2(\varepsilon^{-1})>0$ and $W$ is convex at infinity.
%Thus, it is possible to construct self-similar solution such that the function $W$ starting out as concave and decreasing at $y=\varepsilon$, ends up being convex and increasing at $y=\varepsilon^{-1}$ for $\varepsilon$ small enough. That means that there exists point $y=y_1$ such that $W'(y_1)=0$, i.e.\ $W$ achieves its minimum value at $y_1.$  }

\medskip

\begin{example}
On the Figure \ref{fig1}, one can see the solution to the system (\ref{ode}) for $A=0$ obtained using Runge Kutta method with the initial data at $\varepsilon=0.01$,
\[W(\varepsilon)=-5, \quad W'(\varepsilon)=-100, \quad R(\varepsilon)=500.\]
%We choose $W'(\varepsilon)<0$ since the numerical experiments have shown that such a condition is physical. Namely, even if we take $W'(\varepsilon)>0$, the monotonicity of $W$ changes after one iteration and the solution continues to behave as in the case $W'(\varepsilon)<0.$ 
\end{example}
\begin{center}
\begin{figure}
\includegraphics[scale=0.8]{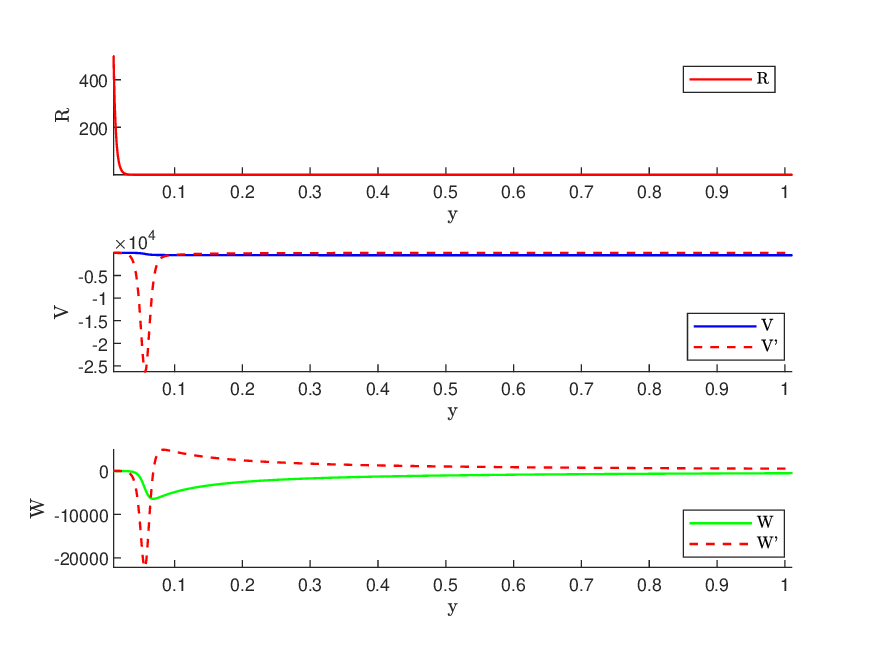}
\caption{Numerical solution to (\ref{ode}) for $A=0$ with $\varepsilon=0.01$, $r_0=500,\omega_0=-5, \omega_1=-100.$ }
\label{fig1}
\end{figure}
\end{center}

\bigskip
%{\color{red}
Numerical experiments confirmed  that self-similar solution to (\ref{ode}), with the initial conditions at the point $y=\varepsilon$ taken such that assumptions of Theorem \ref{thm:sdw} and initial conditions (\ref{id_collapse}) hold, will not break down for $y>\varepsilon$. The function $W$ starting out as decreasing and concave at the point $y=\varepsilon$,  will reach its minimum value at the point $y=z>\varepsilon$, after which it will start to increase while approaching $W=-1$. On the other side, the function $R$ will quickly decrease to 0, as shown in Proposition \ref{thm:Rvac}.
That can be explained in the following way: The function $W$ starting as concave and decreasing, will preserve such behavior until $|W(y)|$ becomes large enough at the point $y=y_{c}<z$ such that $W''(y_c)=0.$ The first term on the right-hand side of $(\ref{ode})_2$ will dominate as long as $W(y)$ does not increase enough for $W''$ to become equal to zero again (at the point $y=y_d>z$), after which positive $W'$ will start to decrease (see Figure \ref{fig:sketch}). 

\medskip

%{\color{blue} Ovo mi se ne svidja, a mozda je i suvi\v sno.}\\
%{\color{red}
%It is easy to prove that the solution $(R,W)$ to the system (\ref{ode}) such that (\ref{id_collapse}) holds at $y=\varepsilon$ and (\ref{assym_inf}) holds when $y\to \infty$ will exhibit the above explained behavior: In order to get $W\to -1$ when $y\to \infty$, negative, decreasing and concave function $W$ will have to start to increase (at $y=z$). That means that negative and decreasing $W'$ will  start to increase ($W''$ turns from negative to positive) near the point $y_c<z$. Since $W'>0$ and $W''>0$ at least for some $y>z$, $W'$ will  start to decrease in order to achieve $W'\to 0$ when $y\to \infty.$ %}
%For $y>y_d$ we have $W''(y)<0.$ 
%}

\medskip

\begin{figure}

\tikzset{every picture/.style={line width=0.75pt}} %set default line width to 0.75pt        
\begin{tikzpicture}[x=0.75pt,y=0.75pt,yscale=-1,xscale=1]
%uncomment if require: \path (0,208); %set diagram left start at 0, and has height of 208

%Curve Lines [id:da7070072996688102] 
\draw [color={rgb, 255:red, 65; green, 117; blue, 5 }  ,draw opacity=1 ][line width=1.5]    (139.25,90.31) .. controls (173.93,98.82) and (175.75,151.94) .. (212.09,126.64) ;
%Curve Lines [id:da27577059899341183] 
\draw [color={rgb, 255:red, 65; green, 117; blue, 5 }  ,draw opacity=1 ][line width=1.5]    (212.09,126.64) .. controls (250.26,103.04) and (248.44,82.8) .. (298.41,80.27) ;
%Curve Lines [id:da08716927664915008] 
\draw [color={rgb, 255:red, 208; green, 2; blue, 27 }  ,draw opacity=1 ][line width=1.5]    (141.06,146.8) .. controls (181.2,260.7) and (164.85,101.35) .. (194.83,44.02) ;
%Straight Lines [id:da18135040935258706] 
\draw [color={rgb, 255:red, 155; green, 155; blue, 155 }  ,draw opacity=1 ][line width=0.75]  [dash pattern={on 4.5pt off 4.5pt}]  (163.03,44.02) -- (162.12,188.19) ;
%Straight Lines [id:da865582334653616] 
\draw [color={rgb, 255:red, 155; green, 155; blue, 155 }  ,draw opacity=1 ] [dash pattern={on 4.5pt off 4.5pt}]  (194.83,44.02) -- (192.1,132.55) ;
%Curve Lines [id:da4644960784707588] 
\draw [color={rgb, 255:red, 208; green, 2; blue, 27 }  ,draw opacity=1 ][line width=1.5]    (194.83,44.02) .. controls (242.08,-0.67) and (280.24,13.67) .. (298.41,32.22) ;
%Straight Lines [id:da5284687486529636] 
\draw [color={rgb, 255:red, 155; green, 155; blue, 155 }  ,draw opacity=1 ] [dash pattern={on 4.5pt off 4.5pt}]  (258.43,12.82) -- (257.52,92.08) ;
%Rounded Rect [id:dp6215633669602558] 
\draw  [fill={rgb, 255:red, 155; green, 155; blue, 155 }  ,fill opacity=0.09 ][dash pattern={on 0.84pt off 2.51pt}][line width=0.75]  (298.25,36.77) .. controls (298.25,30.06) and (303.69,24.63) .. (310.39,24.63) -- (458,24.63) .. controls (464.7,24.63) and (470.14,30.06) .. (470.14,36.77) -- (470.14,73.19) .. controls (470.14,79.9) and (464.7,85.33) .. (458,85.33) -- (310.39,85.33) .. controls (303.69,85.33) and (298.25,79.9) .. (298.25,73.19) -- cycle ;
%Straight Lines [id:da7304405541410488] 
\draw    (278.03,157.09) -- (383.95,157.09) ;
%Straight Lines [id:da7439800204006526] 
\draw  [dash pattern={on 4.5pt off 4.5pt}]  (278.66,169.48) -- (384.58,169.48) ;
%Curve Lines [id:da4382007665239054] 
\draw [color={rgb, 255:red, 208; green, 2; blue, 27 }  ,draw opacity=1 ][line width=1.5]    (282.31,144.69) .. controls (305.04,151.63) and (347.97,156.09) .. (373.85,155.1) ;
%Curve Lines [id:da22198378599409585] 
\draw [color={rgb, 255:red, 65; green, 117; blue, 5 }  ,draw opacity=1 ][line width=1.5]    (285.36,189.77) .. controls (310.61,174.89) and (345.44,172.46) .. (371.96,171.96) ;
%Rounded Rect [id:dp6299312011052409] 
\draw  [color={rgb, 255:red, 155; green, 155; blue, 155 }  ,draw opacity=1 ][dash pattern={on 0.84pt off 2.51pt}] (253.27,142.31) .. controls (253.27,135.28) and (258.97,129.58) .. (266,129.58) -- (377.45,129.58) .. controls (384.48,129.58) and (390.18,135.28) .. (390.18,142.31) -- (390.18,180.52) .. controls (390.18,187.55) and (384.48,193.25) .. (377.45,193.25) -- (266,193.25) .. controls (258.97,193.25) and (253.27,187.55) .. (253.27,180.52) -- cycle ;

%Straight Lines [id:da9097198284467596] 
\draw    (457.03,157.93) -- (562.95,157.93) ;
%Straight Lines [id:da13013476401558832] 
\draw  [dash pattern={on 4.5pt off 4.5pt}]  (457.66,170.33) -- (563.58,170.33) ;
%Curve Lines [id:da3239809997683869] 
\draw [color={rgb, 255:red, 208; green, 2; blue, 27 }  ,draw opacity=1 ][line width=1.5]    (461.3,145.53) .. controls (516.48,177.23) and (522.84,156.15) .. (550.09,161.21) ;
%Curve Lines [id:da600548369552085] 
\draw [color={rgb, 255:red, 65; green, 117; blue, 5 }  ,draw opacity=1 ][line width=1.5]    (464.35,190.61) .. controls (471.05,145.19) and (504.66,174.7) .. (548.28,169.64) ;
%Rounded Rect [id:dp5310515255575081] 
\draw  [color={rgb, 255:red, 155; green, 155; blue, 155 }  ,draw opacity=1 ][dash pattern={on 0.84pt off 2.51pt}] (432.26,143.16) .. controls (432.26,136.12) and (437.97,130.42) .. (445,130.42) -- (556.44,130.42) .. controls (563.47,130.42) and (569.17,136.12) .. (569.17,143.16) -- (569.17,181.36) .. controls (569.17,188.39) and (563.47,194.09) .. (556.44,194.09) -- (445,194.09) .. controls (437.97,194.09) and (432.26,188.39) .. (432.26,181.36) -- cycle ;

%Straight Lines [id:da3926653022862199] 
\draw [color={rgb, 255:red, 144; green, 19; blue, 254 }  ,draw opacity=1 ]   (274.79,134.23) -- (361.38,72.74) ;
\draw [shift={(363.83,71)}, rotate = 144.62] [fill={rgb, 255:red, 144; green, 19; blue, 254 }  ,fill opacity=1 ][line width=0.08]  [draw opacity=0] (10.72,-5.15) -- (0,0) -- (10.72,5.15) -- (7.12,0) -- cycle    ;
\draw [shift={(274.79,134.23)}, rotate = 324.62] [color={rgb, 255:red, 144; green, 19; blue, 254 }  ,draw opacity=1 ][fill={rgb, 255:red, 144; green, 19; blue, 254 }  ,fill opacity=1 ][line width=0.75]      (0, 0) circle [x radius= 3.35, y radius= 3.35]   ;
%Straight Lines [id:da331109453362749] 
\draw [color={rgb, 255:red, 144; green, 19; blue, 254 }  ,draw opacity=1 ]   (441.06,131.7) -- (403.55,70.19) ;
\draw [shift={(401.99,67.63)}, rotate = 58.63] [fill={rgb, 255:red, 144; green, 19; blue, 254 }  ,fill opacity=1 ][line width=0.08]  [draw opacity=0] (10.72,-5.15) -- (0,0) -- (10.72,5.15) -- (7.12,0) -- cycle    ;
\draw [shift={(441.06,131.7)}, rotate = 238.63] [color={rgb, 255:red, 144; green, 19; blue, 254 }  ,draw opacity=1 ][fill={rgb, 255:red, 144; green, 19; blue, 254 }  ,fill opacity=1 ][line width=0.75]      (0, 0) circle [x radius= 3.35, y radius= 3.35]   ;
%Straight Lines [id:da03130080936065838] 
\draw [color={rgb, 255:red, 128; green, 128; blue, 128 }  ,draw opacity=1 ]   (138.5,179.76) -- (139.39,19.2) ;
\draw [shift={(139.41,16.2)}, rotate = 90.32] [fill={rgb, 255:red, 128; green, 128; blue, 128 }  ,fill opacity=1 ][line width=0.08]  [draw opacity=0] (8.93,-4.29) -- (0,0) -- (8.93,4.29) -- cycle    ;
%Straight Lines [id:da3806939412420769] 
\draw [color={rgb, 255:red, 128; green, 128; blue, 128 }  ,draw opacity=1 ]   (92.16,44.86) -- (488.03,44.86) ;
\draw [shift={(491.03,44.86)}, rotate = 180] [fill={rgb, 255:red, 128; green, 128; blue, 128 }  ,fill opacity=1 ][line width=0.08]  [draw opacity=0] (8.93,-4.29) -- (0,0) -- (8.93,4.29) -- cycle    ;
%Straight Lines [id:da9602994456079696] 
\draw [color={rgb, 255:red, 128; green, 128; blue, 128 }  ,draw opacity=1 ] [dash pattern={on 4.5pt off 4.5pt}]  (93.52,55.82) -- (480.13,55.82) ;

% Text Node
\draw  [color={rgb, 255:red, 144; green, 19; blue, 254 }  ,draw opacity=1 ][fill={rgb, 255:red, 189; green, 16; blue, 224 }  ,fill opacity=0.05 ]  (267.32, 142.57) circle [x radius= 13.6, y radius= 13.6]   ;
\draw (261.32,134.97) node [anchor=north west][inner sep=0.75pt]  [color={rgb, 255:red, 144; green, 19; blue, 254 }  ,opacity=1 ]  {$I$};
% Text Node
\draw  [color={rgb, 255:red, 144; green, 19; blue, 254 }  ,draw opacity=1 ][fill={rgb, 255:red, 189; green, 16; blue, 224 }  ,fill opacity=0.05 ]  (446.77, 143.42) circle [x radius= 13.9, y radius= 13.9]   ;
\draw (440.27,135.82) node [anchor=north west][inner sep=0.75pt]  [color={rgb, 255:red, 144; green, 19; blue, 254 }  ,opacity=1 ]  {$II$};
% Text Node
\draw (188.67,24.38) node [anchor=north west][inner sep=0.75pt]    {$z$};
% Text Node
\draw (154.83,22.54) node [anchor=north west][inner sep=0.75pt]    {$y_{c}$};
% Text Node
\draw  [color={rgb, 255:red, 255; green, 255; blue, 255 }  ,draw opacity=1 ][fill={rgb, 255:red, 255; green, 255; blue, 255 }  ,fill opacity=1 ]  (246.09,16.98) -- (270.09,16.98) -- (270.09,42.98) -- (246.09,42.98) -- cycle  ;
\draw (249.09,21.38) node [anchor=north west][inner sep=0.75pt]    {$y_{d}$};
% Text Node
\draw (289.07,6.83) node [anchor=north west][inner sep=0.75pt]  [color={rgb, 255:red, 208; green, 2; blue, 27 }  ,opacity=1 ]  {$W'$};
% Text Node
\draw (283.8,86.08) node [anchor=north west][inner sep=0.75pt]  [color={rgb, 255:red, 65; green, 117; blue, 5 }  ,opacity=1 ]  {$W$};
% Text Node
\draw  [color={rgb, 255:red, 255; green, 255; blue, 255 }  ,draw opacity=1 ][fill={rgb, 255:red, 255; green, 255; blue, 255 }  ,fill opacity=1 ]  (120.16,28.41) -- (138.16,28.41) -- (138.16,50.41) -- (120.16,50.41) -- cycle  ;
\draw (123.16,32.81) node [anchor=north west][inner sep=0.75pt]  [font=\small]  {$0$};
% Text Node
\draw  [color={rgb, 255:red, 255; green, 255; blue, 255 }  ,draw opacity=1 ][fill={rgb, 255:red, 255; green, 255; blue, 255 }  ,fill opacity=1 ]  (112.57,46.43) -- (137.57,46.43) -- (137.57,68.43) -- (112.57,68.43) -- cycle  ;
\draw (115.57,50.83) node [anchor=north west][inner sep=0.75pt]  [font=\small]  {$-1$};
% Text Node
\draw  [color={rgb, 255:red, 255; green, 255; blue, 255 }  ,draw opacity=1 ][fill={rgb, 255:red, 255; green, 255; blue, 255 }  ,fill opacity=1 ]  (479.15,46.27) -- (497.15,46.27) -- (497.15,70.27) -- (479.15,70.27) -- cycle  ;
\draw (482.15,50.67) node [anchor=north west][inner sep=0.75pt]    {$y$};
\end{tikzpicture}
\vspace{-1.6cm}
\caption{Behavior of functions $W$ and $W'$ for $y>\varepsilon$}
\label{fig:sketch}
\end{figure}
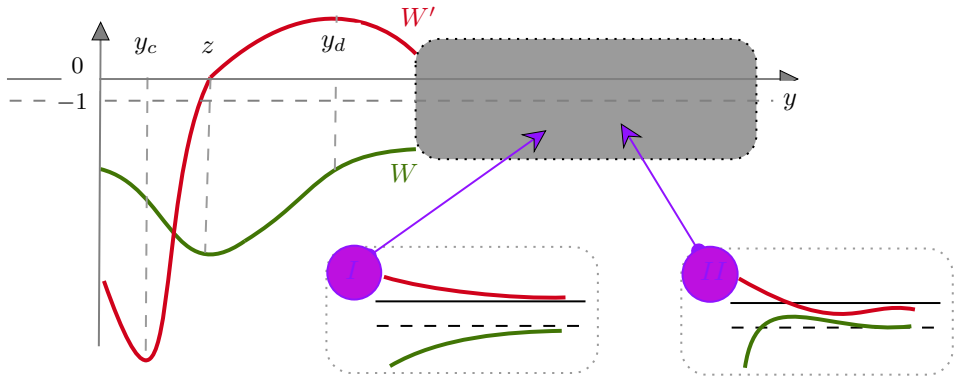

Even thought it has been proved that $W<-\frac{1}{3}$,  it is left to investigate under which conditions it is possible to obtain the lower bound $W\leq -1$. However, it can be proved that there exists $y_d>z$ such that $W(y)<-1$ for $y\in(\varepsilon,y_d)$ and
\[W''(y_d)=0,\, W'(y_d)>0,\, W(y_d)<-1.\]
%The behavior of $W$ and $W''$ after the point $y=y_d$ stays as open question since 
\begin{proposition}
Let assumptions of  Theorem \ref{thm:main_anl} hold and define $y_d>z>\varepsilon$ such that
\[W''(y_d)=0,\, W'(y_d)>0,\, W(y_d)<-1.\]
Then $W(y)<-1$ for $y\in (\varepsilon,y_d].$
\end{proposition}
\begin{proof}
We know that $W'(y)\leq 0$ and as a consequence $W(y)\leq W(\varepsilon)<-1$ for $y\in[\varepsilon,z]$, so it is left to prove that $W(y)<-1$ on the interval $(z,y_d).$
Suppose that there exists $b\in (z, y_d)$ such that $W(b)=-1$ (we know that $W'(y)>0$ and $W''(y)>0$ on $(z, y_d)$). From the equation (\ref{ode}) we have
\[-b^2W''(b)=\frac{1}{2}b^2(W'(b)b-R(b))+\underbrace{(W'(b)b+1)^2-1}_{>0 \text{ since } W'(b)>0}-\frac{H(b)}{2}>0,\]
which implies that $W''(b)<0$. Hence, the function $W$ cannot reach value $-1$ for $y\leq y_d.$ 
\end{proof}

Behavior of the functions $W$ and $W'$ for $y>y_d$ is not completely explained. Most numerical experiments show that the function $W$ will continue to increase and it will never pass the line $W=-1$ ($W(y)<-1$ and $W'(y)>0$ for $y>y_d$). However, the second scenario (see Figure \ref{fig:sketch}) when  increasing $W$ crosses the line  $W=-1$ at the point $y_e>y_d$, while $W'$ becomes negative (and $W$ decreasing) for $y>y_f>y_e$ is not yet rejected. It stays unclear if the results that will support that scenario are obtained as a consequence of numerical error. However, one should have in mind that even if $W$ becomes larger than $-1$ at some point, the following
\[H(y)=W'(y)y+3W+1<0, \quad W(y)<-\frac{1}{3}\]
will hold due to Theorem \ref{thm:main_anl}. Hence, we have
\[W'(y)<-\frac{3W+1}{y}<\frac{2}{y} \quad \text{for } -1<W(y)<-\frac{1}{3}.\]

%\medskip

\end{document}